\documentclass[11pt]{amsart}
\usepackage{amsmath,amssymb,mathrsfs, xcolor}
\usepackage{graphicx} 
\usepackage[title]{appendix}
\usepackage{geometry}
\usepackage{comment}
\usepackage[normalem]{ulem}
\newcommand{\RN}[1]{%
  \textup{\uppercase\expandafter{\romannumeral#1}}%
}

\newtheorem{cor}{Corollary}[section]
\newtheorem{lem}{Lemma}[section]
\newtheorem{prop}{Proposition}[section]

\newtheorem{defn}{Definition}[section]
\newtheorem{rem}{Remark}[section]
\newtheorem{ex}{Example}[section]

\newtheorem{problem}{Problem}[section]
\newtheorem{conjecture}[problem]{Conjecture}
\newtheorem{question}[problem]{Question}
\newtheorem{thm}[problem]{Theorem}

\title[Scalar curvature comparison and rigidity of weakly convex domains]{Scalar curvature comparison and rigidity of $3$-dimensional weakly convex domains}
\author{Dongyeong Ko and Xuan Yao}

\address{Department of Mathematics, Rutgers University - New Brunswick, Piscataway, NJ 08854}
\email{dk954@math.rutgers.edu}
\address{Department of Mathematics, Cornell University, Ithaca, NY, 14850}
\email{xy346@cornell.edu}

\begin{document}

\begin{abstract}
For a compact Riemannian $3$-manifold $(M^{3}, g)$ with mean convex boundary which is diffeomorphic to a weakly convex compact domain in $\mathbb{R}^{3}$, we prove that if scalar curvature is nonnegative and the scaled mean curvature comparison $H^{2}g \ge H_{0}^{2} g_{Eucl}$ holds, then $(M,g)$ is flat. Our result is a smooth analog of Gromov's dihedral rigidity conjecture and an effective version of extremity results on weakly convex balls in $\mathbb R^3$. More generally, we prove the comparison and rigidity theorem for several classes of manifold with corners. Our proof uses capillary minimal surfaces with prescribed contact angle together with the construction of foliation with nonnegative mean curvature and with prescribed contact angles.
\end{abstract}

\maketitle
\section{Introduction}

Geometric curvature problems to understand the geometric and topological properties of manifolds with global curvature conditions have been a fundamental topic in differential geometry. In particular, comparison and rigidity results in the context of interaction between interior curvature and boundary curvature have led the progress of such understanding. As initiated by Alexandrov \cite{alexandrov1951}, the notion of metric with sectional curvature lower bound was introduced via geodesic triangle comparison theorem, which was established for Riemannian manifolds by Toponogov \cite{toponogov1957, toponogov1959, toponogov1964} (See \cite{cheeger1997structure}, \cite{cheeger2000structure}, \cite{cheeger2000structure3}, \cite{colding2012sharp}, \cite{cheeger2013lower} on Cheeger-Colding-Naber theory for Ricci curvature lower bounded metric).

While less is known for metrics with scalar curvature lower bound, Gromov \cite{gromov2014dirac} conjectured the Riemannian dihedral rigidity comparison theorem, which can be directly proven by Toponogov type comparison theorem in dimension $2$.
\begin{conjecture}[Gromov]
    Let $P$ be a  convex $n$-dimensional polyhedron in $\mathbb R^n$, equip $P$ with Riemannian metric $g$ satisfying:
    \begin{enumerate}
        \item each face is mean convex with respect to $g$;
        \item the dihedral angles between faces with respect to $g$ is no larger than the diheral angles with respect to the Euclidean metric $g_0$;
        \item $R_g\geq 0$,
    \end{enumerate}
    then the metric $g$ must be flat.
\end{conjecture}
 Gromov proved the conjecture for cubical polyhedra by introducing doubling approach and reduced to torus rigidity theorem (see \cite{schoen1979existence} and \cite{gromov1980spin}), and sketched a proof via harmonic spinors; see \cite{gromov2014dirac}, \cite{gromov2021lecturesscalarcurvature}, \cite{gromov2023convexpolytopesdihedralangles}. Notably, Li proved the conjecture for the conical and prism-type polyhedron applying the variational method with capillary minimal surfaces; see \cite{li2020polyhedron}, \cite{li2024dihedral}.  Wang-Xie-Yu \cite{wang2021gromov} developed spin geometry on cornered manifold to prove Gromov's conjecture. The conjecture has also been proven under extra angle assumptions via harmonic spinor methods and artful smooth-out procedures; see Brendle \cite{brendle2023matchingangle}, Brendle-Wang \cite{brendle2023gromovsrigiditytheorempolytopes}. Recently, Bär-Brendle-Chow-Hanke \cite{bar2023rigidity} proved the rigidity theorems for initial data sets via the spinor approach, which we learned that they introduced a similar boundary condition after finishing our work. We refer to Gromov \cite{gromov2024convex} for smooth-out procedures.

In contrast to dihedral rigidity problems deal with the singular mean curvature generated by its dihedral angle, in the context of Gauss-Bonnet theorem in dimension $2$, it is natural to consider interactions between Scalar curvature and `smooth' mean curvature in smooth domains in dimension $3$ or higher.

 A natural question arises:
 \begin{question}
     What is the corresponding boundary comparison condition for smooth domains in $\mathbb R^3$?
 \end{question}

In this direction, Gromov \cite{gromov2021lecturesscalarcurvature} showed the following extremality for the standard unit ball. We denote $H$ and $H_{0}$ to be the mean curvature of $\partial M$ with respect to $g$ and the Euclidean metric $g_{Eucl}$, respectively. Moreover, we denote $\sigma$ and $\overline{\sigma}$ to be induced metrics of $g$ and $g_{Eucl}$ on $\partial M$, respectively.
\begin{thm}[\cite{gromov2021lecturesscalarcurvature}, Section 5.8.1] Assume that $(M^{3},g)$ be a connected compact manifold with nonnegative scalar curvature, the boundary $\partial M$ of $(M^{3},g)$ is diffeomorphic to a standard $2$-sphere, and satisfies the comparisons of the mean curvature $H \ge H_{0}$ and of the induced metrics $\sigma \ge \overline{\sigma}$, then $(M,g)$ is flat.
\end{thm}

Recently, Chai-Wang \cite{CW1} proved comparison and rigidity type theorem with boundary condition $H \ge H_{0}$ and $\sigma \ge \overline{\sigma}$ on $\mathbb{S}^{1}$-rotationally symmetric setting using capillary minimal surface (See \cite{chai2024dihedral} for rigidity in hyperbolic setting). See Wu \cite{wu2024capillary} for recent capillary minimal surface approach on nonnegative scalar curvature metric in different setting. We refer to Shi-Tam \cite{shi2002positive} for a total mean curvature comparison theorem under the existence of boundary isometric embedding into Euclidean spaces i.e. $\sigma = \overline{\sigma}$. 

For the recent construction result of critical point of perturbed area functional such as prescribed mean curvature surface and capillary surface -- with min-max approach, we refer to Zhou-Zhu \cite{zhou2018existence} and Li-Zhou-Zhu \cite{li2021min}.

Our main result is the following comparison and rigidity result:

\begin{thm}[Main Theorem]\label{comparison}
Let $(M^{3},g)$ be a connected compact manifold,  $M$ is diffeomorphic to a weakly convex domain in $\mathbb{R}^{3}$. Suppose scalar curvature on $M$ satisfies $R_g \ge 0$, and has a mean convex boundary and mean curvature satisfies 
\[
H^{2}g\ \ge H_{0}^{2} g_{Eucl},
\]
where the metrics are restricted to $T(\partial M)$. 

Assume that one of the following cases is satisfied:
\begin{enumerate}
\item $\partial M$ is smooth;
\item $\partial M=\Sigma_1\cup \Sigma\cup \Sigma_2$, where $\Sigma_1$ and $\Sigma_2$ are smooth sub-domains of parallel planes, $\Sigma$ is also smooth and the dihedral angles between the surfaces at intersection points with respect to $g$ is no larger than that with respect to $g_{Eucl}$, we call such $M$ as generalized prism type;
\item $\partial M\setminus \{ p\}$ is smooth and $p$ is a conical point satisfying \textit{conical comparison}.
\end{enumerate}
Then $g$ is Euclidean.
\end{thm}

\begin{rem}
    After completing this work, we were informed by the authors of \cite{wang2021gromov}, that an equivalent boundary comparison condition was introduced in \cite{wang2021gromov}.
\end{rem}

We use the Euclidean ball example to explain more about our boundary comparison condition.
\begin{ex}
Consider Euclidean ball of radius $r$, the mean curvature of its boundary with respect to the Euclidean metric is $\frac{n-1}{r}$, where $n$ is the dimension of the Euclidean space. 
\end{ex}
From this example, we see that it may not be reasonable to propose merely mean curvature comparison on the boundary since the mean curvature itself is not scaling invariant. If we aim to propose a boundary comparison that, together with non-negative scalar curvature ensuring Euclidean rigidity, we should propose a scaled mean curvature comparison condition which is scaling invariant.

\begin{rem}
  Another possible scaling invariant boundary comparison assumption in dimension $3$ is $H_g^2da_g\geq H_0^2da_{g_{Eucl}}$. However, we do not know whether this boundary comparison condition together with non-negative scalar curvature condition gives Euclidean rigidity. In fact, there is little room to weaken the boundary comparison condition we propose for the local estimates for stability operator of capillary minimal surfaces. Moreover, considering the higher dimensional generalization, our boundary comparison condition may make more sense.
\end{rem}

\begin{defn}[Conical comparison]
    For a conical point $p$, we say it satisfies conical comparison if at $p$ for any $z_1,z_2\in T_pM$,
    \[
    \frac{\langle z_1,z_2\rangle_g}{|z_1|_g|z_2|_g}\geq \frac{\langle z_1,z_2\rangle_{g_{Eucl}}}{|z_1|_{g_{Eucl}}|z_2|_{g_{Eucl}}}
    \]
    $T_pM$ is the tangent cone at $p$.
\end{defn}

\begin{rem} $ $
  \begin{enumerate}
      \item Our results hold for weakly convex domains, where Gromov posed the convexity assumption in his conjecture;
      \item The conical comparison is a generalization of the dihedral angle comparison at the vertex.
  \end{enumerate}
\end{rem}

The first key ingredient in our proof is the local comparison estimate of boundary terms from the stability operator for capillary minimal surfaces. We obtain the derivation of local comparison by applying mean curvature comparison and decomposition trick of the mean curvature term of the boundary in Lemma \ref{btlowerbound}. Then we obtain the global lower bound of boundary terms by applying a winding number argument; see Section \ref{sec:localcomparison}. 

We expect that this local comparison and the tricks behind have further applications in other geometric problems where the contact angles are not constant and the manifolds are without symmetry assumption. Moreover, since our approach based on capillary minimal surfaces does not rely on the spin structure of manifolds, the generalization of the arguments via capillary minimal surface will lead to obtain comparison and rigidity results in a larger class of manifolds, in particular, in higher dimension settings.

The main tool to prove the smooth case is a local construction of nonnegative mean curvature foliation with prescribed contact angle under our scaled mean curvature comparison. This is a generalization of Chai-Wang's construction \cite{CW1} in $\mathbb{S}^{1}$-symmetry cases.

\subsection{Sketch of the Proof}

We describe the outline of the proof of Theorem \ref{comparison} mainly on the smooth case (1), which can be generalized to other cases. We analyze a capillary minimal surface which achieves a nontrivial minimizer of a capillary area functional, whose prescribed contact angle function is determined by the Euclidean domain $(M,g_{Eucl})$. We foliate the Euclidean domain with flat planes that are perpendicular to the $z$-axis, and extract the prescribed angle function $\bar{\rho}\in(0,\pi)$ on $\partial M$ as an intersection angle between leaves and $\partial M$ from the Euclidean domain. Then we define the capillary functional by
\begin{equation*}
    A^{\overline{\rho}}(\Omega) =  |\partial^{i}\Omega| - \int_{\partial^{b}\Omega} \cos \overline{\rho} d \mathcal{H}^{2},
\end{equation*}
where $\Omega \subseteq M$ is an open domain, and $\partial^{i}\Omega$ and $\partial^{b}\Omega$ are the interior and boundary part of $\partial M$, respectively.

We take a minimizer $\Omega$ of the functional $A^{\overline{\rho}}$, corresponding $\Sigma = \partial^{i}\Omega$, and the normal part of the variation by $f$. Then by the stability of minimizer, we have the following stability inequality:
\begin{equation*}
 - \int_{\Sigma} (f \Delta f + (|A|^{2}+ \text{Ric}(N,N))f^{2}) d\mathcal{H}^{2} + \int_{\partial \Sigma} f \Big( \frac{\partial f}{\partial \eta}- Qf \Big) \ge 0,
\end{equation*}
where $Q$ is defined to be
\begin{equation*}
    Q = \frac{1}{\sin \overline{\rho}} A_{\partial M} (\nu,\nu) - \cot \overline{\rho} A(\eta,\eta) + \frac{1}{\sin^{2} \overline{\rho}} \partial_{\nu} \cos \overline{\rho}.
\end{equation*}
After plugging $f \equiv 1$ into the stability inequality and simplifying the boundary terms by applying the Schoen-Yau trick, our local comparison estimate (Lemma \ref{btlowerbound}) gives control of the boundary term. 

The key idea of the local comparison estimate is projecting the boundary terms, written in terms of $\nu$ and $\tau$ which are normal and tangent vector field of $\partial \Sigma$ on $\partial M$, to those on the boundary of standard Euclidean slices, hence in $\overline{\tau}$ and $\overline{\nu}$. We apply the boundary mean curvature comparison to properly decomposed terms and obtain the following inequality. See Lemma \ref{btlowerbound} for details.
\begin{equation*}  
    ( H(p) - \nabla_{\nu} \overline{\rho} ) \ge  \langle \tau, (\nabla_{\overline{\tau}} \overline{\rho}) \overline{\nu}+(H_{0}(p)- \nabla_{\overline{\nu}}\overline{\rho}) \overline{\tau} \rangle_{g_{Eucl}}.
\end{equation*}

Then we can take advantage of integrating this value over a curve $\partial \Sigma$, since the right hand side of the previous inequality consists of terms in Euclidean terms and $\tau$, tangential vector field of $\partial \Sigma$ on $\partial M$. By integration, we obtain a global lower bound by a winding number.

With the global estimate, we prove that a non-trivial minimizer of the capillary functional must be infinitesimally rigid. Constructing a CMC foliation with prescribed capillary angle near the infinitesimally rigid minimal surface, one can prove a local splitting theorem which guarantees the metric to be Euclidean. By connectedness of $M$, the local splitting theorem can be generalized to a global splitting theorem and thus $(M,g)$ must be Euclidean.

Another key part of our proof is to prove the existence of a nontrivial minimizer of $A^{\overline{\rho}}$ on $(M,g)$. In the smooth case, we construct the foliation with a nonnegative mean curvature and prescribed contact angle near a strictly convex point under our scaled mean curvature comparison condition. Once we constructed such a foliation, it serves as a barrier to guarantee the existence of a non-trivial minimizer of $A^{\overline{\rho}}$ and hence we obtain the Euclidean rigidity.

Our technical strategy is to construct another foliation by slicing a small neighborhood of a vertex with leaves with a nonnegative mean curvature, while optimizing a contact angle by keeping a contact angle to be smaller than or equal to the contact angle of the Euclidean foliation. Inspired by Chai-Wang's computations of mean curvature and contact angle, we construct foliations using quadratic graphs. The difficulty arises from the fact that the axes of the local coordinate of $\partial M$ which is an elliptic paraboloid and those induced metric by projection to the first two coordinates may not agree, because of the absence of rotational symmetry. We introduce the foliations with `weighted' quadratic graphs in Section 5 to overcome this technical obstacle by adjusting this discrepancy.

For the generalized prism type, the minimizer is non-trivial since bottom and top face serves as natural barriers by the dihedral angle comparison.

For the conical point case, we first prove that if the conical angle comparison is strict at the conical point then the minimizer is non-trivial (Lemma \ref{nontrivialcone}) and the local estimates for the stability operator guarantees the metric to be flat. Then for the matching conical angle, we construct a CMC foliation with prescribed contact angle which 
are infinitesimally rigid and hence the metric is flat by local splitting theorem.

\subsection{Organization}
 The organization of this paper is as follows. In Section $2$, we introduce the prescribed capillary functional and its stability operator and regularity of minimizers. In Section $3$, we prove the local comparison argument and global lower bound of the boundary term of stability operator. In Section $4$, we prove the rigidity by showing the local splitting theorem. In Section $5$, we prove the nontriviality of the minimizer by constructing local foliations.

 In Appendix A, we prove the computations on Euclidean coordinate.

\subsection{Acknowledgements}
We would like to thank the first author's advisor Prof. Daniel Ketover and the second author's advisor Prof. Xin Zhou for their helpful discussions and encouragements on this project. We are grateful to Prof. Fernando Codá Marques for inspiring conversations. We are also grateful to Prof. Christine Breiner, Prof. Otis Chodosh and Prof. Antoine Song for their interest in our work. We appreciate to Prof. Jonathan Zhu for his interest in our work and pointing out typos. D.K. thanks to Tin Yau Tsang for helpful discussions on the context. D.K. is partially supported by University and Louis Bevier fellowship, NSF grant DMS-1906385 and DMS-2405114 and X.Y. is supported by Hutchinson Fellowship.

\section{Preliminaries}
We consider a compact manifold with boundary $(M^{3},g)$ whose boundary is mean convex and diffeomorphic to a weakly convex sphere in $\mathbb{R}^{3}$. Suppose that $\Sigma \subset M$ is a surface with boundary whose boundary $\partial \Sigma \subset \partial M$. Assume $\Sigma$ separates $\text{int}(M)$ into two connected components. We fix one component containing $p_{+}$ and call this $\Omega$. We denote the outward pointing unit normal vector field of $\partial M$ in $M$ by $X$ and $N$ the unit normal vector field of $\Sigma$ pointing inside $\Omega$. Denote $\nu$ the unit normal vector field of $\partial \Sigma$ in $\partial M$ pointing outward $\Omega$, and $\eta$ the unit normal vector field of $\partial \Sigma$ in $\Sigma$ pointing outward. We define the contact angle, which is a intersecting angle between $\Sigma$ and $\partial M$ by $\rho$. i.e. $\cos \rho = \langle X, N \rangle$. We define an inner product in metric $g$ and $g_{Eucl}$ by $\langle \cdot, \cdot \rangle_{g}$ and $\langle \cdot, \cdot \rangle_{g_{Eucl}}$, respectively.

\subsection{Setup of capillary functional with prescribed angle}

In this section, we define a capillary functional we will work on for the variational problem. We consider $(M,g_{Eucl})$ where $g_{Eucl}$ is an Euclidean metric and take an Euclidean coordinate. Denote $t_{-}$ and $t_{+}$ by a minimum and maximum value of $z$-coordinate, respectively. Then take a contact angle function $\overline{\rho}: \partial M \rightarrow [-\pi, \pi]$ in Euclidean metric by
\begin{equation}
    \cos \overline{\rho} = \Big\langle X, \frac{\partial}{\partial z} \Big\rangle_{g_{Eucl}}.
\end{equation}
We denote $\partial^{b}\Omega$ and $\partial^{i}\Omega$ to be $\partial \Omega \cap \partial M$ and $\partial \Omega \cap \text{int}(M)$, respectively. We set up the prescribed capillary area functional by
\begin{equation}
    A^{\overline{\rho}}(\Omega) =  |\partial^{i}\Omega| - \int_{\partial^{b}\Omega} \cos \overline{\rho} d \mathcal{H}^{2}.
\end{equation}
We consider the minimization problem of $A^{\overline{\rho}}$ on $(M,g)$ as follows:
\begin{equation}\label{eq:minimization}
\mathcal{I} = \inf \{ A^{\overline{\rho}}(\Omega) : \Omega \in \mathcal{E} \},
\end{equation}
where $\mathcal{E}$ is a collection of contractible open sets $E$ such that either one of $\partial M \cap \{ z = t_{+} \} \subset E$ or $\partial M \cap \{ z = t_{-} \} \subset E$ holds. 
\subsection{First and second variational formula} We suppose a relative boundary of solution of the minimization problem is regular and denote $\Sigma = \partial_{rel} \Omega$. We also define
\begin{equation*}
    F(\Sigma) = A^{\overline{\rho}}(\Omega) =  |\Sigma| - \int_{\partial^{b}\Omega} \cos \overline{\rho} d \mathcal{H}^{2}.
\end{equation*}
As a deformation of $\Sigma$, we define $\Psi$ to be a family of diffeomorphisms $\Psi : (-\epsilon, \epsilon) \times \Sigma \rightarrow M$ such that $\Psi(t, \cdot) : \Sigma \rightarrow M$ such that $\Psi(t, \partial \Sigma) \subset \partial M$ and $\Psi(0, \Sigma) = \Sigma$ for $t \in (-\epsilon, \epsilon)$. 

Denote $\Sigma_{t} = \Psi(t,\Sigma)$ and $E_{t}$ as the corresponding component separated by $\Sigma_{t}$. Let $Y = \frac{\partial \Psi(t,\cdot)}{\partial t}|_{t=0}$ be the vector field generating $\Psi$. Note that $Y$ is tangential to $\partial M$ along $\partial \Sigma$. Define $\mathcal{A}(t) = F(\Sigma_{t})$ and $f = \langle Y, N \rangle _{g}$. Then we have the first variation of $\mathcal{A}(t)$ by 
\begin{equation}\label{firstv}
    \mathcal{A}'(0) = - \int_{\Sigma} H f d\mathcal{H}^{2} + \int_{\partial \Sigma} \langle Y, \eta - \nu \cos \overline{\rho}  \rangle_{g} d\mathcal{H}^{1},
\end{equation}
where $H$ is a mean curvature vector of $\Sigma$. We call $\Sigma$ by \emph{capillary minimal surface} if $ \mathcal{A}'(0)=0$ for every admissible deformations. By (\ref{firstv}), $\Sigma$ is a capillary minimal surface if and only if $H \equiv 0$ and $\eta - \nu \cos \overline{\rho}$ is normal to $\partial M$.

Now assume $\Sigma$ is a capillary minimal surface and we consider the second variational formula of prescribed capillary functional:
\begin{equation}\label{secondv}
    \mathcal{A}''(0) = - \int_{\Sigma} (f \Delta f + (|A|^{2}+ \text{Ric}(N,N))f^{2}) d\mathcal{H}^{2} + \int_{\partial \Sigma} f \Big( \frac{\partial f}{\partial \eta}- Qf \Big),
\end{equation}
where $Q$ is defined to be
\begin{equation}
    Q = \frac{1}{\sin \overline{\rho}} A_{\partial M} (\nu,\nu) - \cot \overline{\rho} A(\eta,\eta) + \frac{1}{\sin^{2} \overline{\rho}} \partial_{\nu} \cos \overline{\rho}.
\end{equation}
We call a capillary minimal surface $\Sigma$ is stable if $\mathcal{A}''(0)$ is nonnegative for all admissible deformations. We refer the proof of (2.7) of \cite{CW1} for the proof of (\ref{secondv}). We can change the second variational formula for later use as in \cite{CW1} and \cite{li2020polyhedron}. By applying Lemma 2.1 in \cite{CW1}, we obtain
\begin{equation} \label{rewrite}
    Q = -H \cot \overline{\rho} + \frac{H_{\partial M}}{\sin  \overline{\rho}} - \kappa + \frac{1}{\sin^{2} \overline{\rho}} \partial_{\nu} \cos \overline{\rho},
\end{equation}
where $\kappa$ is a geodesic curvature of $\partial \Sigma$ in $\Sigma$.
\subsection{Regularity of minimizers} In this section, we collect the classical regularity result of $A^{\overline{\rho}}$-minimizers by Taylor \cite{taylor1977boundary} and De Philippis-Maggi \cite{philippis2015regularity} in more general settings. Note that the regularity result is local.
\begin{thm}[Regularity of Capillary minimal surfaces \cite{taylor1977boundary}, \cite{philippis2015regularity}] \label{regularity}
Suppose $M$ is a smooth $3$-manifold with boundary and $U$ is a relatively open set in $M$. Assume that $\Omega$ minimizes $A^{\overline{\rho}}$-functional. Then $\Sigma := \partial_{rel} \Omega$ is an embedded minimal surface in $U$, and properly embedded on the neighborhood of each point $p \in \Sigma \cap \partial M$. Moreover, $\Sigma$ intersects $\partial M$ transversally with prescribed angle $\overline{\rho}$. 
\end{thm}
\begin{rem}
We can prove that the boundary of minimizer $\partial \Sigma \subset \partial M$ does not contain $p_{+}$ in cone and smooth cases in our main Theorem \ref{comparison} by maximum principle. This enables to us to proceed the local estimate arguments in Section 3. We prove that the minimizer is nontrivial in Section 5.
\end{rem}
\section{Local comparison estimate and boundary term estimate}\label{sec:localcomparison}
In this section, we prove the local comparison estimate of boundary terms follows from the stability operator to prove the comparison and rigidity theorem. Our analysis is based on the basic computations in Appendix A. We parametrize a boundary $\partial M$ of $(M,g_{Eucl})$ by $\psi(u,v) = (x(u,v),y(u,v),v)$ where $(u,v) \in S^{1} \times [v_{-},v_{+}] $ and $x_{u}^{2} +y_{u}^{2} = \text{const}$ on level sets $\partial M \cap \{ z= v \}$ for each $v \in [v_{-},v_{+}]$. Denote $\overline{\tau}$ and $\overline{\nu}$ to be the unit tangential vector field on each level set in counterclockwise direction and the unit normal vector field in $(\partial M,g_{Eucl})$ so on Euclidean metric. The estimate in this section will be proven under the assumption of full regularity by Theorem \ref{regularity}. First, we prove an inequality which arises from mean curvature comparison.
\begin{lem} \label{mccomparison} Assume that $H^{2}g \ge H_{0}^{2} g_{Eucl}$ holds. Then for $a, b \in \mathbb{R}$ where not both $a$ and $b$ are zero and $p \in \partial M$, the following inequality holds:
\begin{equation*} 
    H(p) \ge \frac{H_{0}(p)}{a^{2}+b^{2}} \langle a \tau + b \nu, a \overline{\tau} + b \overline{\nu} \rangle_{g_{Eucl}}.
\end{equation*}
\begin{proof}
We apply the mean curvature comparison $H^{2}g \ge H_{0}^{2} g_{Eucl}$ and the decomposition into Euclidean coordinate and obtain the following inequality. Denote $w$ is a unit vector starting from $p$ which is perpendicular to $a \overline{\tau} + b \overline{\nu}$ in Euclidean metric.
\begin{align}
    \nonumber (a^{2}+b^{2}) H^{2} &= H^{2} \langle a \tau + b \nu, a \tau + b \nu \rangle_{g} \\ \label{mccappl} &\ge H_{0}^{2} \langle a \tau + b \nu, a \tau + b \nu \rangle_{g_{Eucl}} \\ \label{coorddecomp} &= H_{0}^{2} \bigg(\bigg\langle a \tau + b \nu, \frac{1}{\sqrt{a^{2}+b^{2}}}(a \overline{\tau} + b \overline{\nu}) \bigg\rangle_{g_{Eucl}}^{2} + \langle a \tau + b \nu, w \rangle_{g_{Eucl}}^{2}\bigg) \\ \nonumber &\ge \frac{H_{0}^{2}}{a^{2}+b^{2}}\langle a \tau + b \nu, a \overline{\tau} + b \overline{\nu} \rangle_{g_{Eucl}}^{2},
\end{align}
where (\ref{mccappl}) comes from the mean curvature comparison condition and (\ref{coorddecomp}) follows from the Euclidean coordinate decomposition. By taking square roots on both hand sides, we obtain the inequality.
\end{proof}
\end{lem}
Now we prove a local comparison estimate of boundary terms arising from the second variational formula.
\begin{lem} \label{btlowerbound} The following inequality holds.
\begin{equation*} 
    ( H(p) - \nabla_{\nu} \overline{\rho} ) \ge  \langle \tau, (\nabla_{\overline{\tau}} \overline{\rho}) \overline{\nu}+(H_{0}(p)- \nabla_{\overline{\nu}}\overline{\rho}) \overline{\tau} \rangle_{g_{Eucl}}
\end{equation*}
If the equality holds at $p$, then one of the the boundary rigidity conditions holds:
\begin{align}
    \bar{H}_{\partial M}^2g_{Eucl}|_{\partial M}=H_{\partial M}^2g|_{\partial M},
\end{align}
or
\begin{align}
    \bar{H}_{\partial M}=\nabla_{\bar{\nu}}\bar{\rho}>0, \quad H_{\partial M}=(\nabla_{\nu}\bar{\rho})>0, \quad \nu\parallel \bar{\nu}.
\end{align}
\begin{proof}
Note that by Lemma \ref{lem:weak convexity} we have the following inequality. 
 \begin{align}\label{eq: weakconvexity}
 (H_{0}-\nabla_{\overline{\nu}}\overline{\rho})\nabla_{\overline{\nu}}\overline{\rho}- (\nabla_{\overline{\tau}}\overline{\rho})^{2} \ge 0,
 \end{align}

 Now we can assume that both $H_{0}(p)$ and $H_{0}(p)-\nabla_{\overline{\nu}}\overline{\rho}$ are not zero. Otherwise, the desired inequality is trivially true. Then we have the following straightforward identity by rearranging terms:
\begin{equation} \label{weight}
   W_{1} +W_{2} := \frac{(\nabla_{\overline{\tau}}\overline{\rho})^{2}+ (H_{0}(p)- \nabla_{\overline{\nu}}\overline{\rho})^{2}}{H_{0}(p)(H_{0}(p)- \nabla_{\overline{\nu}}\overline{\rho})} +\frac{(H_{0}(p)- \nabla_{\overline{\nu}}\overline{\rho})\nabla_{\overline{\nu}}\overline{\rho}- (\nabla_{\overline{\tau}}\overline{\rho})^{2}}{H_{0}(p)(H_{0}(p)- \nabla_{\overline{\nu}}\overline{\rho})} =1.
\end{equation}
We need to check that $W_{1}, W_{2} \ge 0$ where $W_{1} \ge 0$ is straightforward and we can figure out $W_{2} \ge 0$ by \eqref{eq: weakconvexity}. Now we can apply the inequality in Lemma \ref{mccomparison} and (\ref{weight}) and have
\begin{equation} \label{weightedineq}
H \ge W_{1}  \frac{H_{0}}{(H_{0}- \nabla_{\overline{\nu}}\overline{\rho})^{2}+(\nabla_{\overline{\tau}}\overline{\rho})^{2}} \langle (H_{0}- \nabla_{\overline{\nu}}\overline{\rho}) \tau + (\nabla_{\overline{\tau}}\overline{\rho}) \nu, (H_{0}- \nabla_{\overline{\nu}}\overline{\rho}) \overline{\tau} + (\nabla_{\overline{\tau}}\overline{\rho}) \overline{\nu} \rangle_{g_{Eucl}}+ W_{2} H_{0}\langle \nu, \overline{\nu} \rangle_{g_{Eucl}}
\end{equation}
at $p$. Since $\overline{\rho}$ is a function defined on $\partial M$ and in we consider this function as a function defined on pullback metric $(\partial M, g)$, we can derive the following by decomposition of $\nu$ by $\overline{\tau}$ and $\overline{\nu}$.
\begin{equation} \label{angleeucldecomp}
    \nabla_{\nu}\overline{\rho} =  \langle \nu, \overline{\tau} \rangle_{g_{Eucl}} \nabla_{\overline{\tau}}\overline{\rho} + \langle \nu, \overline{\nu} \rangle_{g_{Eucl}} \nabla_{\overline{\nu}}\overline{\rho}. 
\end{equation}
Now we obtain
\begin{align}
    \label{angleeucldecomp2} ( H - \nabla_{\nu} \overline{\rho} ) &\ge ( H - \langle \nu, (\nabla_{\overline{\tau}}\overline{\rho}) \overline{\tau} + (\nabla_{\overline{\nu}}\overline{\rho} )\overline{\nu} \rangle_{g_{Eucl}}) \\ \label{ineqappl} &\ge W_{1}  \frac{H_{0}}{(H_{0}- \nabla_{\overline{\nu}}\overline{\rho})^{2}+(\nabla_{\overline{\tau}}\overline{\rho})^{2}} \langle (H_{0}- \nabla_{\overline{\nu}}\overline{\rho}) \tau + (\nabla_{\overline{\tau}}\overline{\rho}) \nu, (H_{0}- \nabla_{\overline{\nu}}\overline{\rho}) \overline{\tau} + (\nabla_{\overline{\tau}}\overline{\rho}) \overline{\nu} \rangle_{g_{Eucl}}\\ \nonumber & + W_{2} H_{0}\langle \nu, \overline{\nu} \rangle_{g_{Eucl}} - \langle \nu, (\nabla_{\overline{\tau}}\overline{\rho}) \overline{\tau} + (\nabla_{\overline{\nu}}\overline{\rho} )\overline{\nu} \rangle_{g_{Eucl}},
\end{align}
where we obtain (\ref{angleeucldecomp2}) from (\ref{angleeucldecomp}) and (\ref{ineqappl}) from (\ref{weightedineq}). We now verify each coefficient of inner products involving $\nu$ vanishes:
\begin{equation} \label{tauterm}
W_{1}  \frac{H_{0}}{(H_{0}- \nabla_{\overline{\nu}}\overline{\rho})^{2}+(\nabla_{\overline{\tau}}\overline{\rho})^{2}} (H_{0}- \nabla_{\overline{\nu}}\overline{\rho})\nabla_{\overline{\tau}}\overline{\rho} - \nabla_{\overline{\tau}}\overline{\rho} = 0
\end{equation}
and
\begin{align} \label{nuterm}
    W_{1}  \frac{H_{0}}{(H_{0}- \nabla_{\overline{\nu}}\overline{\rho})^{2}+(\nabla_{\overline{\tau}}\overline{\rho})^{2}}(\nabla_{\overline{\tau}}\overline{\rho})^{2} + W_{2} H_{0}- \nabla_{\overline{\nu}}\overline{\rho} &= \frac{(\nabla_{\overline{\tau}}\overline{\rho})^{2}}{(H_{0}- \nabla_{\overline{\nu}}\overline{\rho})} + \frac{(H_{0}- \nabla_{\overline{\nu}}\overline{\rho})\nabla_{\overline{\nu}}\overline{\rho}- (\nabla_{\overline{\tau}}\overline{\rho})^{2}}{H_{0}- \nabla_{\overline{\nu}}\overline{\rho}}- \nabla_{\overline{\nu}}\overline{\rho}  = 0.
\end{align}
By pluging (\ref{tauterm}) and (\ref{nuterm}) into (\ref{ineqappl}) and rearranging terms, we obtained the desired inequality.

Now suppose the equality is achieved, we discuss two cases:
\begin{enumerate}
    \item $H_0(p)=0$: By Lemma \ref{lem:weak convexity} we have 
    \[
    0\leq (H_0(p)-\nabla_{\bar{\nu}}\bar{\rho})\nabla_{\bar{\nu}}\bar{\rho}-(\nabla_{\bar{\tau}}\bar{\rho})^2=-(\nabla_{\bar{\nu}}\bar{\rho})^2-(\nabla_{\bar{\tau}}\bar{\rho})^2,
    \]
    which implies that
    \[
    d\bar{\rho}(p)=0.
    \]
    Then we have
    \[
    H(p)=0.
    \]
    \item $H_0(p)\neq 0$: If $H_0(p)-\nabla_{\bar{\nu}}\bar{\rho} \neq 0$, then we have $\bar{H}_{\partial M}^2g_{Eucl}|_{\partial M}=H_{\partial M}^2g|_{\partial M}$ directly by tracing the equality case of inequalities. Otherwise suppose $H_0(p)-\nabla_{\bar{\nu}}\bar{\rho} = 0$, we obtain
    \[
    \nabla_{\bar{\tau}}\bar{\rho}=0.
    \]
    Furthermore, we have
    \[
    H^2g(\nu,\nu)\geq (\nabla_{\bar{\nu}}\bar{\rho})^2g_{Eucl}(\nu,\nu)^2\geq (\nabla_{\bar{\nu}}\bar{\rho})^2g_{Eucl}(\nu,\bar{\nu})^2=(\nabla_{\nu}\bar{\rho})^2
    \]
    when the equality is achieved, we have
    \[
    \nu\parallel \bar{\nu}.
    \]
\end{enumerate}

\end{proof}
\end{lem}

We prove a simple topological lemma based on in an Euclidean domain $(M,g_{Eucl})$ as follows.
\begin{lem} \label{windingnumber} Suppose $\gamma$ to be a simple closed curve with a parametrization with counterclockwise direction on $\partial M \cap \{v_{-} < z< v_{+} \}$ which separates $\partial M \cap \{ z = v_{-} \}$ and $\partial M \cap \{ z = v_{+} \}$. Then the following holds:
\begin{equation*}
\int_{\gamma} \arctan \bigg( \frac{y_{u}}{x_{u}}\bigg)ds = 2 \pi.
\end{equation*}
\begin{proof}
Note that the integral achieves the value of integer multiple of $2 \pi$ i.e. $2k \pi$ for some $k \in \mathbb{Z}$, so the value is discretized. We can homotope $\gamma$ to a simple closed curve which is very close to $\partial M \cap \{ z= v_{+} \}$ and obtain the integral to be $2 \pi$ by standard winding number argument. Also since the integral is preserved over homotopies by discretized values, we have the integral value to be $2 \pi$.
\end{proof}
\end{lem}
Now we prove the estimate of an integral of boundary term over a boundary curve of surface $\Sigma$ which is a minimizer $\mathcal{I}$, assuming the existence of nontrivial minimizer which we will prove in Section 5. By the regularity theorem (Theorem \ref{regularity}) and the nontriviality assumption, we can assume that all boundary curves are simple closed curves and take one of those boundary components (See \cite{CW1} for the rotationally symmetric setting).
\begin{lem} \label{curvstimate} For a simple closed curve $\gamma$ which is a boundary component of $\Sigma$ on $(M,g)$ and separates $\partial M \cap \{ z = v_{-} \}$ and $\partial M \cap \{ z = v_{+} \}$, the following estimate holds.
\begin{equation*}
\int_{\gamma} \frac{1}{\sin \overline{\rho}} ( H- \nabla_{\nu} \overline{\rho} ) ds \ge 2 \pi.
\end{equation*}
\begin{proof}
We obtain the estimate in terms of arc length parametrization of $\gamma$ by $a$:
\begin{align*}
    \langle \gamma',(\nabla_{\bar{\tau}}\bar{\rho})\bar{\nu}&+(H_0-\nabla_{\bar{\nu}}\bar{\rho})\bar{\tau}\rangle _{g_{Eucl}}\\
    &=cu'(a)(H_0-\nabla_{\bar{\nu}}\bar{\rho})+\langle v'(a)\frac{\partial \psi}{\partial v},\RN 2(\bar{\tau},\bar{\nu})\bar{\nu}+\RN{2}(\bar{\tau},\bar{\tau})\bar{\tau}\rangle\\
    &=cu'(a)(H_0-\nabla_{\bar{\nu}}\bar{\rho})+v'(a)\langle\nabla_{\frac{\partial\psi}{\partial v}}X,\bar{\tau}\rangle\\
    &=cu'(a)(H_0-\nabla_{\bar{\nu}}\bar{\rho})-v'(a)\langle X,\nabla_{\frac{\partial\psi}{\partial v}}\bar{\tau}\rangle\\
    &=cu'(a)(H_0-\nabla_{\bar{\nu}}\bar{\rho})-v'(a)\langle X,\frac{d}{dv}\bar{\tau}\rangle\\
    &=u'(a )\sin\bar{\rho}\frac{x_{uu}y_u-y_{uu}x_u}{x_u^2+y_u^2}+v'(a)\frac{x_{uv}y_u-y_{uv}x_u}{c A}.\\
    &=u'(a )\sin\bar{\rho}\frac{x_{uu}y_u-y_{uu}x_u}{x_u^2+y_u^2}+v'(a)\sin\bar{\rho}\frac{x_{uv}y_u-y_{uv}x_u}{x_u^2+y_u^2}
\end{align*}
Where $c=\sqrt{x_u^2+y_u^2}$, $A=\sqrt{c^2+(x_uy_v-y_ux_v)^2}$, and we used the fact that $\sin\bar{\rho}=\frac{c}{A}$.

Combining Lemma \ref{btlowerbound}, Lemma \ref{windingnumber} we obtain 
\begin{align*}
    \int_{\gamma} \frac{1}{\sin \overline{\rho}} ( H - \nabla_{\nu} \overline{\rho} ) ds &\ge \int_{\gamma} \bigg( \frac{x_{uu}y_{u}-y_{uu}x_{u}}{(x_{u}^{2}+y_{u}^{2})}u'(a)  + \frac{x_{uv}y_{u}-y_{uv}x_{u}}{x_{u}^{2}+y_{u}^{2}} v'(a) \bigg) da \\ 
    &= \int_{\gamma} \bigg( \partial_{u}\bigg( \arctan \bigg( \frac{y_{u}}{x_{u}}\bigg) \bigg)u'(a) + \partial_{v}\bigg( \arctan \bigg( \frac{y_{u}}{x_{u}}\bigg) \bigg)v'(a) \bigg) da \\ 
    &= \int_{\gamma} \arctan \bigg( \frac{y_{u}}{x_{u}}\bigg)ds = 2 \pi.
\end{align*}
\end{proof}
\end{lem}
We analyze the stability inequality of minimizer. We follow the arguments in \cite{li2020polyhedron} and \cite{CW1}, and we obtain equality conditions which we will use for the rigidity analysis.
\begin{thm} \label{ineqproof} Suppose the conditions in Theorem \ref{comparison}. Then the following equality holds on at least one connected component $\Sigma$ of a minimizer.
\begin{align}
    \nonumber\chi(\Sigma)=1,\quad R_{M}=0,\quad \|A\|=0,\quad &K_{\Sigma}=0 \quad \text{on }\Sigma;\\
     \kappa_{\partial \Sigma} = \frac{H_{\partial M}}{\sin  \overline{\rho}} + \frac{1}{\sin^{2} \overline{\rho}} \partial_{\nu} \cos \overline{\rho}, \quad \langle X,N&\rangle_g=\cos\bar{\rho}  \quad \text{on }\partial\Sigma
\end{align}
and one of the the boundary rigidity conditions:
\begin{align}
    \bar{H}_{\partial M}^2g_{Eucl}|_{\partial M}=H_{\partial M}^2g|_{\partial M},
\end{align}
or
\begin{align}
    \bar{H}_{\partial M}=\nabla_{\bar{\nu}}\bar{\rho}>0, \quad H_{\partial M}=(\nabla_{\nu}\bar{\rho})>0, \quad \nu\parallel \bar{\nu}.
\end{align}
Moreover, we call $\Sigma$ satisfying the above rigidity conditions infinitesimally rigid.
\begin{proof} The contact angle condition is directly obtained by the first variation formula. By the second variation formula (\ref{secondv}) on the capillary minimal surface $\Sigma$, we have 
\begin{equation} \label{stabineq}
    - \int_{\Sigma} (f \Delta f + (|A|^{2}+ \text{Ric}(N,N))f^{2}) d\mathcal{H}^{2} + \int_{\partial \Sigma} f \Big( \frac{\partial f}{\partial \eta}- Qf \Big) \ge 0.
\end{equation}
By plugging a constant function $f \equiv 1$ into (\ref{stabineq}), we have
\begin{equation} \label{stabineqconst}
    - \int_{\Sigma} (|A|^{2}+ \text{Ric}(N,N)) d\mathcal{H}^{2} - \int_{\partial \Sigma} \frac{1}{\sin \overline{\rho}} A_{\partial M} (\nu,\nu) - \cot \overline{\rho} A(\eta,\eta) + \frac{1}{\sin^{2} \overline{\rho}} \partial_{\nu} \cos \overline{\rho} \ge 0.
\end{equation}
We apply Gauss equation and (\ref{rewrite}) as in \cite{li2020polyhedron} and \cite{CW1} and simplify the inequality (\ref{stabineqconst}).
\begin{equation} \label{simplifiedstabineq}
    - \int_{\Sigma} \frac{1}{2}(R_{M} - 2 K_{\Sigma}+ |A|^{2}) d\mathcal{H}^{2} - \int_{\partial \Sigma} \Big( \frac{H_{\partial M}}{\sin  \overline{\rho}} - \kappa_{\partial \Sigma} + \frac{1}{\sin^{2} \overline{\rho}} \partial_{\nu} \cos \overline{\rho} \Big) \ge 0.
\end{equation}
By applying Gauss-Bonnet formula to (\ref{simplifiedstabineq}), we obtain
\begin{equation} \label{finalineq}
     - \int_{\Sigma} \frac{1}{2}(R_{M} + |A|^{2}) d\mathcal{H}^{2} - \int_{\partial \Sigma} \frac{1}{\sin \overline{\rho}} ( H_{\partial M}- \nabla_{\nu} \overline{\rho} ) \ge - 2 \pi \chi(\Sigma) \ge -2 \pi.
\end{equation}
Then by applying Lemma \ref{curvstimate} and since $f \equiv 1$ is a Jacobi vector field of equation with Robin boundary condition, we obtain
\begin{equation*}
   \chi(\Sigma)=1,\quad R_{M}=0,\quad \|A\|=0,\quad K_{\Sigma}=0 \quad \text{on }\Sigma \text{ and } \kappa_{\partial \Sigma} = \frac{H_{\partial M}}{\sin  \overline{\rho}} + \frac{1}{\sin^{2} \overline{\rho}} \partial_{\nu} \cos \overline{\rho}\text { on } \partial \Sigma.
\end{equation*}
The boundary conditions are carried out by the proof of Lemma \ref{btlowerbound}.
\end{proof}
\end{thm}
\section{Rigidity}
In this section, we construct CMC foliation near non-trivial minimizer, and we state a local splitting theorem. We refer Ye \cite{ye1991foliation}, Bray-Brendle-Neves \cite{bray2010rigidity} and Ambrozio \cite{ambrozio2015rigidity} for the construction of CMC foliations.
\subsection{Non-trivial minimizer}
Suppose there exists a non-trivial energy minimizer $\Sigma=\partial\Omega\cap \mathring{M}$, we construct a local foliation by CMC capillary surfaces near the minimizer with prescribed capillary angle.

\begin{prop}\label{prop:cmcfoliation}
    Suppose $\Sigma$ is a properly embedded, two-sided, minimal capillary surface in $M^3$, assume further that $\Sigma$ is infinitesimally rigid, then there exists $\epsilon>0$ and a function $w:\Sigma\times (-\epsilon,\epsilon)\to\mathbb R$ such that, for every $t\in(-\epsilon,\epsilon)$, the set 
    \begin{align*}
        \Sigma_t=\{\phi(x,w(x,t)):x\in\Sigma\}
    \end{align*}
    is a capillary surface with constant mean curvature $H(t)$ and it meets $\partial M$ with a contact angle $\bar{\rho}$. Moreover, for every $x\in\Sigma$ and every $t\in(-\epsilon,\epsilon)$,
    \begin{align*}
        w(x,0)=0,\quad \int_{\Sigma}(w(x,t)-t)=0,\quad \frac{\partial w}{\partial t}(x,t)=1,
    \end{align*}
where $\phi(x,t)$ is a flow of a vector field $Y$, $Y$ is defined in a neighborhood of $\Sigma$, and it is tangential when restricted to $\partial M$. By choosing a sufficiently small $\epsilon$, $\{\Sigma_t\}_{(-\epsilon,\epsilon)}$ is a foliation of a neighborhood of $\Sigma_0=\Sigma$ in $M$.
    
\end{prop}

\begin{proof}
    Consider the Banach spaces
    \begin{align*}
        F=\left\{d\in C^{2,\alpha}(\Sigma)\cap C^{1,\alpha}(\bar{\Sigma}):\int_{\Sigma}d=0\right\},\quad G=\left\{d\in C^{0,\alpha}(\Sigma):\int_{\Sigma}d=0\right\}.
    \end{align*}
    Given small $\delta>0$ and $\epsilon>0$, define the map $\Psi:(-\epsilon,\epsilon)\times (B_0(\delta)\subset F)\to G\times C^{1,\alpha}(\partial\Sigma)$ as
    \begin{align*}
        \Psi(t,d)=\left(H_{t+d}-\frac{1}{|\Sigma|}\int_{\Sigma}H_{t+d},\langle N_{t+d},X_{t+d}\rangle_g-\cos\bar{\rho}(\phi(t,d))\right),
    \end{align*}
    where we use the subscript $d$ to denote the quantities associated with the surface $\Sigma_d=\{\phi(x,d(x)):x\in\Sigma\}$.

    In order to apply the inverse function theorem, we need to compute $D\Psi_d\vert_{(0,0)}$. The computations are followed from the first and second variation formula (\ref{firstv}) and (\ref{secondv}).

    \begin{align*}
        D_d\Psi\vert_{(0,0)}(0,m)=\frac{d}{ds}\vert_{s=0}\Psi(0,sm)=(-\Delta_{\Sigma}m+\frac{1}{|\Sigma|}\int_{\partial\Sigma}\frac{\partial m}{\partial\eta},-\sin\bar{\rho}\frac{\partial m}{\partial \eta}).
    \end{align*}

    It now follows from the standard elliptic theory that there exists unique solution to Laplace equation with Neumann boundary condition, and thus $D\Psi(0,0)$ is an isomorphism when restricted to $0\times F$. Then we apply the inverse function theorem and the results follow from standard arguments, see the proof of Proposition 10 in \cite{li2020polyhedron}.%
\end{proof}

Chai-Wang \cite{CW1} proved the following local splitting theorem which applies in our setting.
\begin{thm}[Chai-Wang \cite{CW1}]\label{thm:local splitting}
    For the CMC foliation $\Sigma_t$, there exists a continuous function $\Gamma(t)$ such that
    \[
    \frac{d}{dt}\left(\exp\left(-\int_0^{t}\Gamma(s)ds\right)H(t)\right)\leq 0.
    \]
\end{thm}
\begin{rem}
    To prove Theorem \ref{thm:local splitting}, Chai-Wang only used the fact that $\Sigma_t$ has constant mean curvature and it has the prescribed capillary angle $\bar{\rho}$.
\end{rem}

\begin{cor}\label{cor:infinitesmally}
    Every $\Sigma_t$ constructed in Proposition \ref{prop:cmcfoliation} is infinitesimally rigid.
\end{cor}
\begin{proof}
    Let 
    \[
    F(t)=|\Sigma_t|-\int_{\partial^{b}E_t}\cos\bar{\rho}.
    \]
    By the first variation formula, we have
    \[
    F(t_2)-F(t_1)=\int_{t_1}^{t_2}\left(\int_{\Sigma_t}H(t)\, da\right)dt,
    \]
    for any $-\epsilon<t_1<t_2<\epsilon$.
    
    As a direct corollary of Theorem \ref{thm:local splitting}, we know that $H(t)$ is non-positive when $t\leq 0$, and it is non-negative when $t\geq 0$. 

    Since $\Sigma=\Sigma_0$ is a non-trivial minimizer, we have 
    \[
    H(t)\equiv 0,\qquad \forall t\in(-\epsilon,\epsilon),
    \]
    and hence every $\Sigma_t$ constructed in Proposition \ref{prop:cmcfoliation} is infinitesimally rigid.
\end{proof}

\section{Proof of the Main Theorem}
We divide the proof of the main theorem (Theorem \ref{comparison}) by three types of domains: generalized prism type, conical domain with a trivial minimizer, smooth domain cases. In the previous section, we proved the local splitting theorem in the rigidity case after assuming the existence of the nontrivial minimizer. In the cases of conical and smooth domains, we construct a barrier to guarantee the existence of nontrivial minimizer. In both cases, the barrier of the Riemannian domain is a foliation near a single vertex on the boundary $\partial M$ where the leaves have a non-negative mean curvature and a larger contact angle than the one in the Euclidean domain. This ensures an existence of nontrivial minimizer of the functional $A^{\overline{\rho}}$ in Riemannian domain, which follows from maximum principle. While our construction of foliations is inspired by the work of Chai-Wang \cite{CW1}, we emphasize the technical part of our construction for general classes of domains in this section.
\subsection{Generalized prism type}

\begin{proof}[Proof of the Main Theorem: generalized prism] 

For the generalized prism type, the minimizer of the capillary energy functional is non-trivial. Suppose there exists a non-trivial minimizer $\Sigma$ of the capillary energy functional, then $\Sigma_0$ is infinitesimally rigid and by Proposition \ref{prop:cmcfoliation}, there exists a CMC foliation $\{\Sigma_t\}_{(-\epsilon,\epsilon)}$ with prescribed contact angle $\bar{\rho}$.

By the local splitting Theorem \ref{thm:local splitting} and standard ODE arguments which were also used in Li \cite{li2020polyhedron}, Chai-Wang \cite{CW1}, we conclude that the CMC foliation are all infinitesimally rigid and by connectedness of $M$, this is a global splitting thus $g$ is Euclidean.

\end{proof}

\subsection{Rigidity case of conical domains}
We now construct a CMC foliation near the conical point under the trivial minimizer assumption.

\begin{lem} \label{nontrivialcone}
    Suppose $M$ is a compact manifold with boundary satisfying assumptions in Theorem \ref{comparison}, $p$ is the conical vertex, then if there exists $z_i\in T_pM$ $i=1,2$, such that
    \begin{align}\label{strictconeangle}
        \frac{\langle z_1,z_2\rangle_g}{|z_1|_g\cdot|z_2|_g}>\frac{\langle z_1,z_2\rangle_{g_{Eucl}}}{|z_1|_{g_{Eucl}}\cdot |z_2|_{g_{Eucl}}}
    \end{align}
    Then the minimizer of \eqref{eq:minimization} is non-trivial.
\end{lem}
\begin{proof}
    Assume for contradiction that for any $\Omega\subset M$,
    \begin{align}\label{eq:conecontradiction}
      A^{\bar{\rho}}(\Omega)=|\partial^i\Omega|-\int_{\partial^b\Omega}\cos\bar{\rho}d\mathcal{H}^2\geq 0.
    \end{align}
    Since we are discussing the generalized cone case, $\bar{\rho}(u,v)=\bar{\rho}(u)$ near the conical point, i.e. the inequality \eqref{eq:conecontradiction} is scaling invariant. More precisely, if $\Omega$ satisfies \eqref{eq:conecontradiction}, then under the homothety $\mu_r: x\to r(x-p)$, the set $(\mu_r)_{\#}(\Omega)\subset (\mu_r)_{\#}(M)$ satisfies \eqref{eq:conecontradiction}. Letting $r\to\infty$, the tangent cone $T_pM$ shares the same property.

    By the assumption of the angle comparison, the tangent cone of $M$ at $p$ can be placed inside the tangent cone of the Euclidean model. Since there exists $z_i\in T_pM$, such that 
    \[
      \frac{\langle z_1,z_2\rangle_g}{|z_1|_g\cdot|z_2|_g}>\frac{\langle z_1,z_2\rangle_{g_{Eucl}}}{|z_1|_{g_{Eucl}}\cdot |z_2|_{g_{Eucl}}},
    \]
    $T_pM$ is strictly contained in the Euclidean cone for some non-empty set $U$. Without loss of generality, we can write the Euclidean cone model as $\psi(u,v)=(vx(u),vx(u),v)$ for $u\in[0,2\pi]$, and $W=\{\psi(u,v):u\in[0,\theta_0]\}\subseteq U$. Then there exists a plane $\pi\in\mathbb R^3$, such that $\pi$ meets the boundary of Euclidean cone model at $\bar{\rho}(\psi(u,v))$, and it meets the boundary of the tangent cone of $M$ at $p$ with angle $\rho(\psi(u,v))$, where
    \[
    \cos\bar{\rho}(\psi(u,v))\leq \cos\rho(u,v)
    \]
    and the inequality is strict when $u\in[0,\theta_0]$.

    Denote $\Omega_{\infty}$ the open domain enclosed by $\pi$ and the boundary of the Euclidean cone model $C$, then since we are using the Euclidean metric, by the area formula, we have that
    \[
    \mathcal{H}^2(\pi\cap\partial \Omega_{\infty})-\int_{\partial C\cap\partial\Omega_{\infty}}\cos\bar{\rho}=0.
    \]
    Since $\mathcal{H}^1(W\cap \partial C\cap \partial\Omega_{\infty})\neq 0$, we conclude that
    \[
    \mathcal{H}^2(\pi\cap\partial\Omega_{\infty})-\int_{\partial C\cap\partial\Omega_{\infty}}\cos\rho<0,
    \]
    contradiction achieved.
\end{proof}

As a direct corollary, if the minimizer is trivial, then at the conical point, we must have 
\begin{align}\label{eq: eqanglevertex}
\frac{\langle z_1,z_2\rangle_g}{|z_1|_g\cdot |z_2|_g}=\frac{\langle z_1,z_2\rangle_{g_{Eucl}}}{|z_1|_{g_{Eucl}}\cdot |z_2|_{g_{Eucl}}},\quad \forall z_1,z_2\in T_pM.
\end{align}

Now, under the equal angle vertex assumption \eqref{eq: eqanglevertex}, we construct a CMC foliation near the vertex $p$.

\begin{thm}\label{thm:CMCnearvertex}
    Let $(M^3,g)$ be a compact manifold with boundary diffeomorphic to a weakly convex domain sphere $M'$ in $\mathbb R^3$ but at $p$ which is a conical vertex. Assume further that $(T_pM,g)$ is isometric to $(C(p'), g_0)$, where $C(p')$ is the tangent cone of $M'$ at $p'$. Then there exists a small neighborhood $U$ of $p$ in $M$, such that $U$ is foliated by $\{\Sigma_{h}\}_{h\in (0,\epsilon)}$ such that:
    \begin{enumerate}
        \item for each $h\in(0,\epsilon)$, $\Sigma_h$ meets $\partial M$ at the prescribed angle $\bar{\rho}$;
        \item each $\Sigma_h$ has constant mean curvature $\lambda_h$, and $\lim_{h\to 0}\lambda_h=0$.
    \end{enumerate}
\end{thm}
\begin{rem}
    Our proof follows Li's proof of Theorem 4.3 in \cite{li2020polyhedron} with minor modification.
\end{rem}

\begin{proof}
    Let $\pi$ be a plane in $\mathbb R^3$ such that $\pi$ meets $C(p')$ at the prescribed angle $\bar{\rho}$. For $h\in (0,1]$, let $\pi_h$ be the plane that is parallel to $\pi$ and the distance between $\pi_h$ and $p'$ is $h$.

    Denote $C(p)$ the tangent cone of $M$ at $p$, $X$ the outward pointing unit vector field of $\partial C(p)$, $N_h$ the unit normal vector field of $\Sigma_h$ pointing toward $p$, $Y(x)$ a vector field such that $Y(x)$ is parallel to $(x-p)$ $\forall x\in\Sigma_h$ and we require the flow of $Y$, denoted as $\phi(x,t)$, parallel translates $\{\Sigma_h\}$. For $x\in \partial C(p)$, $Y(x)$ is tangential to $\partial C(p)$ by the definition of a cone.

    For $d\in C^{2,\alpha}(\Sigma_1)\cap C^{1,\alpha}(\bar{\Sigma}_1)$, we define the perturbed surface:
    \[
    \Sigma_{h,d}=\{\phi(hx,d(hx)):x\in\Sigma_1\}.
    \]
     $\Sigma_{h,d}$ is a small perturbation of $\Sigma_h$ provided $|d|$ is small.

    We use the $h$ subscript to denote the geometric quantities related to $\Sigma_h$ and $(h,d)$ subscript to denote the geometric quantities related to $\Sigma_{h,d}$, both with respect to the $g$ metric.

    \begin{lem}[Ambrozio \cite{ambrozio2015rigidity}, Li, \cite{li2020polyhedron}]\label{lem: Taylorexpansion}
    We have the following Taylor expansion of the geometric quantities:
    \begin{align}
         H_{h,d}=H_h+\frac{1}{h^2}\Delta_hd+(\text{Ric}(N_h,N_h)+|A_h|^2)d+L_1d+Q_1(d)
    \end{align}
    and
    \begin{align}
        \nonumber\langle X_{h,d},N_{h,d}\rangle&=\langle X_h,N_h\rangle -\frac{\sin\bar{\rho}(\phi(hx,0))}{h}\frac{\partial d}{\partial \eta_h}\\
        &+(\cos\bar{\rho}(\phi(hx,0))A(\eta_h,\eta_h)+\RN{2}(\nu_h,\nu_h))d+L_2d+Q_2(d).
    \end{align}
        Where $Q_1,Q_2$ are at least quadratic in $d$, $L_1,L_2$ are uniformly bounded.
    \end{lem}

    Denote $D_h=\langle X_h,N_h\rangle-\cos\bar{\rho}(\phi(hx,0))$, with Lemma \ref{lem: Taylorexpansion}, we deduce the prescribed constant mean curvature $\lambda$ surface equation as:
    \begin{equation}
        \begin{cases}
            \Delta_hd+d^2L_1d+h^2Q_1(d)=h^2(\lambda-H_h) & \text{ in } \Sigma_1\\
            \frac{\partial d}{\partial\eta_h}=hD_h+hL_2d+hQ_2(d) & \text{ on } \partial \Sigma_1.
        \end{cases}
    \end{equation}
    Similarly as in Proposition \ref{prop:cmcfoliation}, we consider the Banach spaces
    \begin{align*}
        F=\left\{d\in C^{2,\alpha}(\Sigma_1)\cap C^{1,\alpha}(\bar{\Sigma}_1):\int_{\Sigma_1}d=0\right\},\quad G=\left\{d\in C^{0,\alpha}(\Sigma_1):\int_{\Sigma_1}d=0\right\}.
    \end{align*}
    For sufficiently small $\delta>0$, define $\Psi(-\epsilon,\epsilon)\times (B_\delta(0)\subset F)\to G\times C^{1,\alpha}(\partial\Sigma_1)$ as
    \begin{align*}
        \Psi(h,d)=\left(\mathcal{L}_h(d)-\frac{1}{|\Sigma_1|}\int_{\Sigma_1}\mathcal{L}_h(d)d\mathcal{H}^2,\mathcal{B}_h(d)\right),
    \end{align*}
    where $\mathcal{L}_h$ and $\mathcal{B}_h$ are operators defined as
     \begin{align*}
        \begin{cases}
            \mathcal{L}_h(d)=\Delta_hd-h^2L_1d-h^2Q_1(d)+h^2H_h,\\
            \mathcal{B}_h(d)=\frac{\partial d}{\partial\eta_h}-hD_h-hL_2d-hQ_2(d).
        \end{cases}
    \end{align*}
    Note that $\Delta_h$ converges to the Laplacian operator on $\mathbb{R}^2$ as $h\to 0$, the linearized operator $D_d\Psi\vert_{(0,0)}$ when restricted to $\{0\}\times F$, is given by 
    \[
    D_d\Psi\vert_{(0,0)}(0,m)=(\Delta m-\int_{\partial\Sigma_1}\frac{\partial m}{\partial\eta},\frac{\partial m}{\partial\eta}).
    \]
    It is an isomorphism and thus we apply the inverse function theorem and obtain that for $h\in (-\epsilon,\epsilon)$, where $\epsilon$ is a sufficiently small number, there exists a $C^1$ map between $h\in(-\epsilon,\epsilon)\to d_{h}(x)\in B_{\delta}(0)\subset F$ for every $h\in(-\epsilon,\epsilon)$, such that
    \[
    \Psi(h,d_h(x))=(0,0).
    \]
    We obtain a foliation of CMC surfaces meeting $\partial M$ at prescribed angle $\bar{\rho}$. Since $d_0=0$ by definition, we denote $m=\frac{\partial d_h}{\partial h}\vert_{h=0}$, and it is easily checked $m$ is a harmonic function satisfies the Neumann boundary condition and it is also zero. We then conclude that 
    \[
    |d_h|_{1,\alpha,\bar{\Sigma}_1}=o(h),
    \]
    $\Sigma_{h,d_h}$ is a CMC foliation near the vertex $p$.
    We now investigate the mean curvature more precisely.
    \begin{align*}
        \lambda_h=\frac{1}{h}\int_{\partial\Sigma_1}D_h+L_2d_h+Q_2(d_h)+\int_{\Sigma_1}L_1d_h+Q_1(d_h)+H_h+o(1),
    \end{align*}
    $\lambda_h\to 0$ as $h\to 0$, the proof is completed.
    
\end{proof}

\begin{proof}[Proof of Main Theorem: Conical point] $ $

    Suppose $M$ satisfies the strict cone angle comparison (\ref{strictconeangle}), then by Lemma \ref{nontrivialcone}, there exists a non-trivial minimizer $\Sigma$ of the capillary energy functional. Then $\Sigma_0$ is infinitesimally rigid and by Proposition \ref{prop:cmcfoliation}, there exists a CMC foliation $\{\Sigma_t\}_{(-\epsilon,\epsilon)}$ with prescribed contact angle $\bar{\rho}$.

    By the local splitting Theorem \ref{thm:local splitting} and standard ODE arguments which were also used in Li \cite{li2020polyhedron}, Chai-Wang \cite{CW1}, we conclude that the CMC foliation are all infinitesimally rigid and by connectedness of $M$, this is a global splitting thus $g$ is Euclidean.

    If the minimizer is trivial, then by Theorem \ref{thm:CMCnearvertex}, we have a CMC foliation whose mean curvature of leaves $\lambda_h$ converges to $0$ as $h \rightarrow 0$, and we apply arguments of the local splitting Theorem \ref{thm:local splitting} and obtain the same conclusion.
\end{proof}
\begin{rem}
    We omit details here since the arguments are the same as in Li \cite{li2020polyhedron} and Chai-Wang \cite{CW1}.
\end{rem}

\subsection{Smooth cases} In this section, we prove the existence of barrier with nonnegative mean curvature and contact angle whose lower bound is that of Euclidean case, which guarantees the existence of nontrivial minimizer in cases both rigid and nonrigid cases. We generalize Section 5 in \cite{CW1} to all smooth weakly convex domain cases. We take the reference point first which achieves strictly positive principal curvatures in Euclidean metric.

\begin{lem} \label{strictlyconvexpt}
    There exists $p_+\in\partial M$, such that $\bar{H}_{\partial M}(p_+)>0$ and $\det\RN{2}(p_+)>0$.
\end{lem}

\begin{proof}
    By the Gauss equation together with the Gauss-Bonnet formula, the Lemma follows directly.
\end{proof}

We take prescribed contact angle $\overline{\rho}$ for Euclidean metric again if necessary by slicing with planes which are parallel with $T_{p_{+}} \partial M$ in $(M, g_{Eucl})$. By applying the translation if necessary, we assume $p_{+}$ to be an origin in our coordinate.

Let us consider the setup. We rotate the coordinate if necessary that principal curvature at $p_{+}$ has a principal curvature in the directions $\frac{\partial}{\partial x_{1}}$ and $\frac{\partial}{\partial x_{2}}$. We put $g= g_{0} + th + O(t^{2})$ where $g_{0}$ is a constant metric and $h$ is a bounded symmetric $2$-tensor. We put $g_{0} = (a_{ij})_{1 \le i,j \le 3}$, $a_{12}=a_{21}=0$ after rotating if necessary and take $(a^{ij})$ as an inverse matrix of $g_{0}$. Moreover, denote $e_{i} = \frac{\partial}{\partial x^{i}}$ and $e^{i} = a^{ij} e_{j}$ and we write $x = (x_{1},x_{2})$.

By the Implicit function theorem, we can write $\partial M$ to be
\begin{equation} \label{iftparametrization}
    \partial M = \{ (x_{1},x_{2}, -c_{11}x_{1}^{2} - 2c_{12} x_{1}x_{2} - c_{22} x_{2}^{2} + O(|x|^{3})  \text{ : for } x \text{ near } p_{+}\}.
\end{equation}

\begin{cor}
    We have 
    \[
    c_{11}c_{22}>c_{12}^2.
    \]
    \begin{proof}
We obtain the strict inequality by applying Lemma \ref{strictlyconvexpt}. 
    \end{proof}
\end{cor}

Now we construct a foliation with nonnegative mean curvature and contact angle which is larger or than equal to the Euclidean foliations, which will be a main technical part of this section. This gives the existence of barrier to ensure the existence of nontrivial minimizer of prescribed capillary functional.

We denote $B = \sqrt{a^{33}((\sqrt{a_{11}}c_{22}+ \sqrt{a_{22}}c_{11})^{2} + (\sqrt{a_{11}}- \sqrt{a_{22}})^{2}c_{12}^{2})} $ and we define
\begin{align*}
    b_{11} &= a^{33}B^{-1}c_{11}^{-1}(a_{11} (c_{11}c_{22}-c_{12}^{2}) + \sqrt{a_{11}a_{22}} (c_{11}^{2}+c_{12}^{2}))\\
    b_{12} &= a^{33}B^{-1}\sqrt{a_{11}a_{22}}(c_{11}+c_{22})\\ 
    b_{22} &= a^{33}B^{-1}c_{22}^{-1}(a_{22} (c_{11}c_{22}-c_{12}^{2}) + \sqrt{a_{11}a_{22}} (c_{12}^{2}+c_{22}^{2})) .
\end{align*} Also, we say $b_{21}=b_{12}$ and $c_{21}=c_{12}$. For $s,t>0$, let us consider the function $G_{s,t}$ to be
\begin{equation*}
    G_{s,t}(x_{1},x_{2}) = c_{11}(b_{11}(1+s)-1)x_{1}^{2} + c_{22}(b_{22}(1+s)-1)x_{2}^{2} + 2 c_{12} (b_{12}(1+s)-1) x_{1}x_{2} - t^{2}.
\end{equation*}
We define $\Sigma_{s,t} : = \{ (x_{1},x_{2},G_{s,t}(x_{1},x_{2})): (x_{1},x_{2}) \in D \} \cap M$ where $D$ is a unit disk for small $t>0$. Moreover, we can define a set $E \subset \mathbb{R}^{2}$ by
\begin{equation*}
    E_{s} := \{ x \in \mathbb{R}^{2} : c_{11} b_{11} x_{1}^{2} + c_{22}b_{22} x_{2}^{2} + 2 c_{12} b_{12} x_{1}x_{2} <(1+s)^{-1} \}
\end{equation*}
and note that $\frac{1}{t} \Sigma_{s,t} \rightarrow E_{s}$ as $t \rightarrow 0$. Hence as in \cite{CW1} and for small $t>0$, we write $\Sigma_{s,t}$ as a map $E_{s} \rightarrow \Sigma_{s,t}$ such that 
\begin{equation*}
    \Sigma_{s,t} = \{ (tx_{1}+O(t^{3}),tx_{2}+O(t^{3}), t^{2} \{ c_{11}(b_{11}(1+s)-1)x_{1}^{2} + 2c_{12}(b_{12}(1+s)-1) x_{1}x_{2} + c_{22}(b_{22}(1+s)-1) x_{2}^{2} \} + O(t^{4})  : x \in E \}
\end{equation*}
We denote the mean curvature at $x \in \Sigma_{s,t}$ of $\Sigma_{s,t}$ by $H^{g_{0}}_{s,t}(x)$. 

We first compute the difference between contact angles of $\Sigma_{s,t}$ and $\partial M$ between on metric $g_{0}$ and Euclidean metric. Denote the contact angles between $\Sigma_{s,t}$ and $\partial M$ at $x \in \Sigma_{s,t}$ by $\rho_{s,t}^{g}(x)$ and corresponding Euclidean prescribed angle $\overline{\rho}$ by $\overline{\rho}_{s,t}$. We generalize computations in the proof of Proposition 5.16 and Corollary 5.17 in \cite{CW1}. We will use the computation of contact angle $\rho_{s,t}$ when the mean curvature of the foliation $\Sigma_{0,t}$ converges to strictly positive number, and we simplify $\rho_{t} := \rho_{0,t}$ only otherwise. We denote $\Sigma_{t} := \Sigma_{0,t}$ as well in this case.

\begin{prop} \label{cadifference} We have the following comparison of contact angles 
\begin{equation*}
    \cos \rho^{g}_{s,t}(x) = \cos \overline{\rho}_{s,t}(x) - 4st^{2} \bigg\{ \frac{(x_{1}c_{11}b_{11}+ x_{2}c_{12}b_{12})^{2}}{a_{11}a^{33}}+ \frac{(x_{1}c_{12}b_{12}+ x_{2}c_{22}b_{22})^{2}} {a_{22}a^{33}}\bigg\} +O(s^{2}) + O(t^{3})
\end{equation*}
\begin{proof}
We denote the normal vector of $\Sigma_{s,t}$ and $\partial M$ at $x$ by $v_{1}$ and $v_{2}$, respectively.
Then we have
\begin{align*}
v_{1} &= \{ -2tx_{\beta}c_{\alpha\beta}(b_{\alpha\beta}(1+s)-1) +O(t^{3})\}e^{\alpha} + (1+O(t^{2}))e^{3} \\
v_{2} &= ( -2tx_{\beta}c_{\alpha\beta}+O(t^{3})) e^{\alpha} + (1+O(t^{2}))e^{3}.
\end{align*}
We write
\begin{equation*}
    w:= v_{1}-v_{2} = \{ -2tx_{\beta}c_{\alpha\beta}b_{\alpha\beta}(1+s) +O(t^{3})\}e^{\alpha} + O(t^{2}).
\end{equation*}
Then we have as in \cite{CW1} as
\begin{align}
\nonumber \cos \measuredangle_{g_{0}} (\Sigma_{s,t}, \partial M) &= 1- \frac{g_{0}(w,w)g_{0}(v_{1},v_{2}) - g_{0}(w,v_{1})g_{0}(w,v_{2})}{2 g_{0}(v_{1},v_{2})^{2}} +O(t^{4}) \\ \label{angleg0}
&= 1 - 2 \frac{(x_{\beta}c_{\alpha\beta}b_{\alpha\beta}(1+s))^{2}}{a_{\alpha\alpha}a^{33}}t^{2} + O(t^{3}).
\end{align}
We can calculate the angle in Euclidean metric as well:
\begin{equation} \label{angleeucl}
\cos \overline{\rho}_{s,t}(x) = 1- 2t^{2}\{ (x_{1}c_{11}+x_{2}c_{12})^{2} + (x_{1}c_{12}+x_{2}c_{22})^{2} \} + O(t^{3}).
\end{equation}
By computing difference between (\ref{angleg0}) and (\ref{angleeucl}), we have
\begin{align} \label{angleg0eucl}
    \cos \rho^{g_{0}}_{s,t}(x) &- \cos \overline{\rho}_{s,t}(x) \\ \nonumber
    &= 2t^{2} \bigg\{ \bigg( c_{11}^{2}+ c_{12}^{2} - \frac{b_{11}^{2}c_{11}^{2}(1+s)^{2}}{a_{11}a^{33}} - \frac{b_{12}^{2}c_{12}^{2}(1+s)^{2}}{a_{22}a^{33}} \bigg) x_{1}^{2}  \\ \nonumber &+ \bigg(2 c_{12}(c_{11}+c_{22})- \frac{2b_{11}b_{12}c_{11}c_{12}(1+s)^{2}}{a_{11}a^{33}} - \frac{2b_{12}b_{22}c_{12}c_{22}(1+s)^{2}}{a_{22}a^{33}} \bigg) x_{1}x_{2} \\ \nonumber &+ \bigg( c_{12}^{2}+ c_{22}^{2} - \frac{b_{12}^{2}c_{12}^{2}(1+s)^{2}}{a_{11}a^{33}} - \frac{b_{22}^{2}c_{22}^{2}(1+s)^{2}}{a_{22}a^{33}} \bigg) x_{2}^{2} \bigg\}+O(t^{3}).
\end{align}
Now we compute
\begin{align} \nonumber 
c_{11}^{2}&+ c_{12}^{2} - \frac{b_{11}^{2}c_{11}^{2}(1+s)^{2}}{a_{11}a^{33}} - \frac{b_{12}^{2}c_{12}^{2}(1+s)^{2}}{a_{22}a^{33}} \\ 
&= c_{11}^{2}+ c_{12}^{2} - a^{33}B^{-2}(\sqrt{a_{11}} (c_{11}c_{22}-c_{12}^{2}) + \sqrt{a_{22}} (c_{11}^{2}+c_{12}^{2}))^{2}(1+s)^{2} \label{cal1stterm} \\ \nonumber &- a^{33}B^{-2} a_{11} c_{12}^{2}(c_{11}+c_{22})^{2}(1+s)^{2} \\ &= \nonumber c_{11}^{2}+c_{12}^{2} - a^{33}B^{-2} ((\sqrt{a_{11}}c_{22}+ \sqrt{a_{22}}c_{11})^{2} + (\sqrt{a_{11}}- \sqrt{a_{22}})^{2}c_{12}^{2}) \\ \nonumber &(c_{11}^{2}+c_{12}^{2})(1+s)^{2} \\  \nonumber &= \label{cal1stterm} -2s(c_{11}^{2}+c_{12}^{2}) +O(s^{2})
\end{align}
and
\begin{equation} \label{cal2ndterm}
    c_{12}^{2}+ c_{22}^{2} - \frac{b_{12}^{2}c_{12}^{2}(1+s)^{2}}{a_{11}a^{33}} - \frac{b_{22}^{2}c_{22}^{2}(1+s)^{2}}{a_{22}a^{33}} = -2s(c_{12}^{2}+ c_{22}^{2} ) +O(s^{2}).
\end{equation}
Also we can compute
\begin{align}
     \nonumber c_{12}(c_{11}&+c_{22})- \frac{b_{11}b_{12}c_{11}c_{12}(1+s)^{2}}{a_{11}a^{33}} - \frac{b_{12}b_{22}c_{12}c_{22}(1+s)^{2}}{a_{22}a^{33}} \\ \nonumber &= c_{12}(c_{11}+c_{22}) - a^{33}B^{-2}\sqrt{a_{22}}c_{12}(\sqrt{a_{11}} (c_{11}c_{22}-c_{12}^{2}) + \sqrt{a_{22}} (c_{11}^{2}+c_{12}^{2}))(c_{11}+c_{22})(1+s)^{2} \\ \nonumber &- a^{33}B^{-2}\sqrt{a_{11}}c_{12}(\sqrt{a_{22}} (c_{11}c_{22}-c_{12}^{2}) + \sqrt{a_{11}} (c_{12}^{2}+c_{22}^{2}))(c_{11}+c_{22})(1+s)^{2} \\ \nonumber &= c_{12}(c_{11}+c_{22})- a^{33}B^{-2}(a_{11}(c_{12}^{2}+c_{22}^{2})+ a_{22}(c_{11}^{2}+c_{12}^{2}) + 2 a_{11}^{\frac{1}{2}}a_{22}^{\frac{1}{2}}(c_{11}c_{22}-c_{12}^{2}))c_{12}(c_{11}+c_{22})(1+s)^{2} \\ &= \label{cal3rdterm}-2s c_{12}(c_{11}+c_{22}) +O(s^{2}).
\end{align}
By plugging (\ref{cal1stterm}), (\ref{cal2ndterm}) and (\ref{cal3rdterm}) into (\ref{angleg0eucl}), we obtain
\begin{align}
     \nonumber \cos \rho^{g_{0}}_{s,t}(x) - \cos \overline{\rho}_{s,t}(x) &= - 4st^{2}((c_{11}^{2}+c_{12}^{2})x_{1}^{2}+2c_{12}(c_{11}+c_{22}) x_{1}x_{2}+(c_{12}^{2}+ c_{22}^{2} )x_{2}^{2}) +O(s^{2})+O(t^{3}) \\ \label{squareangle}&= - 4st^{2} \bigg\{ \frac{(x_{1}c_{11}b_{11}+ x_{2}c_{12}b_{12})^{2}}{a_{11}a^{33}}+ \frac{(x_{1}c_{12}b_{12}+ x_{2}c_{22}b_{22})^{2}} {a_{22}a^{33}}\bigg\} +O(s^{2}) + O(t^{3}),
\end{align}
where we obtain (\ref{squareangle}) by calculations in (\ref{cal1stterm}), (\ref{cal2ndterm}) and (\ref{cal3rdterm}). By applying Corollary 5.17 in \cite{CW1}, we obtain the formula in $g$ metric and the desired formula.
\end{proof}
\end{prop}
We have the following direct generalization of mean curvature formula of parametrized surface of Lemma 5.22 in \cite{CW1}.

\begin{lem} \label{mcsmoothpara}
Suppose $\Sigma_{t}$ is a family of surfaces which is given by
\begin{equation*} 
    \Sigma_{t} = \{ (tx_{1}+O(t^{3}),tx_{2}+O(t^{3}), t^{2}(c_{11}x_{1}^{2} + 2c_{12} x_{1}x_{2} + c_{22} x_{2}^{2}) + O(t^{4})  \text{ : for } x \text{ near } p_{+} \text{ and } x \in R \}
\end{equation*}
where $R$ is an open set in $\mathbb{R}$ and $A>0$. Then for small $t>0$, the mean curvature $H_{\sigma_{t}}$ under $g_{0}$ is
\begin{equation*}
    H_{\Sigma_{t}}(x) = - \frac{1}{ \sqrt{a^{33}}} \bigg( \frac{2c_{11}}{a_{11}} + \frac{2c_{22}}{a_{22}}\bigg) + t L(x_{1},x_{2})+ O(t^{3}),
\end{equation*}
where $L$ is a linear term bounded uniformly. 
\begin{proof}
For $\alpha \in \{1, 2\}$, the basis of a tangent plane $\Sigma_{t}$ at $x$ would be
\begin{equation*}
    X_{\alpha}= Ae_{\alpha} + 2 c_{\alpha\beta} t x_{\beta} e_{3} + O(t^{2}) e_{\beta} + O(t^{2}) e_{3}
\end{equation*}
with an Einstein notation. The only part we need a modification is to compute $\nabla ^{g_{0}}_{X_{\alpha}} X_{\beta}$:
\begin{equation} \label{covarder}
    \nabla ^{g_{0}}_{X_{\alpha}} X_{\beta} = 2 (Ac_{\alpha\beta}+O(t)) e_{3} + \sum_{\gamma=1}^{2} O(t) e_{\gamma}.
\end{equation}
By plugging (\ref{covarder}) into the mean curvature formula as in the proof of Lemma 5.22 in \cite{CW1}, we obtain the mean curvature formula.
\end{proof}
\end{lem}
By applying Lemma \ref{mcsmoothpara}, we obtain the following mean curvature difference on $\Sigma_{s,t}$ and $\partial M$.

\begin{prop} \label{mcdifference} For small $t>0$, the difference of mean curvature $H^{g_{0}}_{t}$ and $H_{t, \partial M}^{g_{0}}$ at an intersection point $(x,G_{t}(x))$ is
\begin{equation*}
    H^{g_{0}}_{t}(x)- H_{t, \partial M}^{g_{0}} (x) = -  2 \frac{B}{\sqrt{a_{11}a_{22}a^{33}}} + t L(x_{1},x_{2})+ O(t^{2})
\end{equation*}
\begin{proof}
We apply Lemma \ref{mcsmoothpara} to derive 
\begin{equation} \label{mcsmoothfol}
    H_{t}^{g_{0}} (x) = - \frac{1}{ \sqrt{a^{33}}} \bigg( \frac{2c_{11}(b_{11}-1)}{a_{11}} + \frac{2c_{22}(b_{22}-1)}{a_{22}}\bigg) + t L(x_{1},x_{2})+ O(t^{2}).
\end{equation}
Moreover, by (\ref{iftparametrization}) we have
\begin{equation*}
    \partial M = \{ (x_{1},x_{2}, -c_{11}x_{1}^{2} - 2c_{12} x_{1}x_{2} - c_{22} x_{2}^{2} + O(|x|^{3})  \text{ : for } x \text{ near } O\}.
\end{equation*}
Here we can apply Lemma \ref{mcsmoothpara} again and we obtain
\begin{equation} \label{mcbdy} 
H_{t, \partial M}^{g_{0}}(x) = \frac{1}{ \sqrt{a^{33}}} \bigg( \frac{2c_{11}}{a_{11}} + \frac{2c_{22}}{a_{22}}\bigg) + t L(x_{1},x_{2})+ O(t^{3})
\end{equation}
By subtracting (\ref{mcbdy}) from (\ref{mcsmoothfol}), we obtain the following:
\begin{align*}
     H^{g_{0}}_{t}(x)- H_{t, \partial M}^{g_{0}} (x) &=   - \frac{2}{ \sqrt{a^{33}}} \bigg( \frac{c_{11}b_{11}}{a_{11}} + \frac{c_{22}b_{22}}{a_{22}}\bigg) + t L(x_{1},x_{2})+ O(t^{2}) \\ &= - \frac{2\sqrt{a^{33}}}{B} \bigg( \frac{(\sqrt{a_{11}} (c_{11}c_{22}-c_{12}^{2}) + \sqrt{a_{22}} (c_{11}^{2}+c_{12}^{2}))}{\sqrt{a_{11}}} + \frac{(\sqrt{a_{22}} (c_{11}c_{22}-c_{12}^{2}) + \sqrt{a_{11}} (c_{12}^{2}+c_{22}^{2}))}{\sqrt{a_{22}}} \bigg) \\ &+ t L(x_{1},x_{2})+ O(t^{2}) \\ &= - \frac{2\sqrt{a^{33}}}{B} \cdot \frac{B^{2}}{a^{33}\sqrt{a_{11}a_{22}}} + t L(x_{1},x_{2})+ O(t^{2}) \\ &= -  2 \frac{B}{\sqrt{a_{11}a_{22}a^{33}}} + t L(x_{1},x_{2})+ O(t^{2}).
\end{align*}
\end{proof}
\end{prop}

Now we have the following mean curvature comparison on leaves with the Euclidean metric. We denote the mean curvature on $\Sigma_{t}$ and $\partial M$ in metric $g$ by $H_{t}^{g}$ and $H^{g}_{t,\partial M}$, respectively. Moreover, we denote the mean curvature on $\partial M$ on Euclidean metric by $\overline{H}_{t, \partial M}$.
\begin{cor} Suppose $g= g_{0}+th+O(t^{2})$ where $h$ is a bounded symmetric $2$-tensor. Then we have the asymptotic formula of the mean curvature near $O$.
\begin{align*}
    H_{t}^{g}(x) &= H^{g}_{t,\partial M}(x) - \overline{H}_{t, \partial M}(x) +   2(c_{11}+c_{22}) -  2 \frac{B}{\sqrt{a_{11}a_{22}a^{33}}} + t L(x_{1},x_{2})+ O(t^{2})
\end{align*}
\begin{proof}

Given parametrization of $\partial M$ by (\ref{iftparametrization}), we have
\begin{equation} \label{euclideanbdymc}
    \overline{H}_{t, \partial M} = 2c_{11} + 2c_{22}.
\end{equation}
By applying (\ref{euclideanbdymc}) to arguments in Corollary 5.23 in \cite{CW1} directly, we obtain the formula.
\end{proof}
\end{cor} \label{mcestimate}
Now we take $H_{p_{+}} : = \lim_{t \rightarrow 0} H_{t}^{g}(x)$. By taking the limit, we get the formula on mean curvature of leaves where leave approaches to $p_{+}$. 
\begin{cor} We have the following formula on the limit of mean curvature of leaves near $p_{+}$:
\begin{equation*}
    H_{p_{+}} = H^{g}_{\partial M}(O) - \overline{H}_{\partial M}(O) +   2(c_{11}+c_{22}) -  2 \frac{B}{\sqrt{a_{11}a_{22}a^{33}}}.
\end{equation*}
\end{cor}
We construct the foliation after dividing into two cases; One is the case that `scaled' metric at $p_{+}$ restricted to $\{ x_{3} = 0 \}$ is Euclidean i.e. $a_{11} = a_{22}$, and another is all other cases. When $a_{11} \neq a_{22}$, we prove $H_{p_{+}}>0$ first and construct the foliation as Proposition 5.24 in \cite{CW1}.

\begin{prop} \label{positivemeancurv} For $g_{0} = (a_{ij})_{1 \le i,j \le 3}$ with $a_{12}=a_{21}=0$, suppose $a_{11} \neq a_{22}$. Then $H_{p_{+}}>0$.
\end{prop}
\begin{proof}
From $H^{2} g(e_{i},e_{i}) \ge H_{0}^{2} g_{Eucl}(e_{i},e_{i})$ for $i=1,2$, we have
\begin{equation*}
    H \ge \max \bigg\{\frac{1}{\sqrt{a_{11}}}, \frac{1}{\sqrt{a_{22}}} \bigg\} H_{0}. 
\end{equation*}

Without loss of any generality, we assume $a_{11}>a_{22}$ and recall $\overline{H}_{\partial M}(O) = 2(c_{11}+c_{22})$. Hence by Corollary \ref{mcestimate}, we have
\begin{align*}
H_{p_{+}} &= H^{g}_{\partial M}(O) - \overline{H}_{\partial M}(O) +   2(c_{11}+c_{22}) -  2 \frac{B}{\sqrt{a_{11}a_{22}a^{33}}} \\ &\ge \bigg( \frac{1}{\sqrt{a_{22}}} -1 \bigg) \overline{H}_{\partial M}(O) +   2(c_{11}+c_{22}) -  2 \frac{B}{\sqrt{a_{11}a_{22}}a^{33}} \\ &= 2\bigg( \frac{1}{\sqrt{a_{22}}} -1 \bigg) (c_{11}+c_{22}) +   2(c_{11}+c_{22}) -  2 \frac{B}{\sqrt{a_{11}a_{22}a^{33}}} \\ &= 2 \bigg( \frac{(c_{11}+c_{22})}{\sqrt{a_{22}}} - \frac{B}{\sqrt{a_{11}a_{22}a^{33}}} \bigg).
\end{align*}
On the other hand, we have
\begin{align}
\nonumber \bigg( \frac{(c_{11}+c_{22})}{\sqrt{a_{22}}} \bigg)^{2} - \bigg( \frac{B}{\sqrt{a_{11}a_{22}a^{33}}} \bigg)^{2} &= \bigg( \frac{(c_{11}+c_{22})}{\sqrt{a_{22}}} \bigg)^{2} - \frac{a^{33}((\sqrt{a_{11}}c_{22}+ \sqrt{a_{22}}c_{11})^{2} + (\sqrt{a_{11}}- \sqrt{a_{22}})^{2}c_{12}^{2})}{a_{11}a_{22}a^{33}} \\ \nonumber &= c_{11} \bigg( \frac{1}{\sqrt{a_{22}}}-\frac{1}{\sqrt{a_{11}}} \bigg) \bigg(\frac{2c_{22}}{\sqrt{a_{22}}} + \frac{c_{11}}{\sqrt{a_{11}}}+ \frac{c_{11}}{\sqrt{a_{22}}}\bigg) - \frac{(\sqrt{a_{11}}- \sqrt{a_{22}})^{2}c_{12}^{2}}{a_{11}a_{22}} \\ \label{squarecomp} &= \bigg( \frac{1}{\sqrt{a_{22}}}-\frac{1}{\sqrt{a_{11}}} \bigg) \bigg( \frac{2c_{11}c_{22}}{\sqrt{a_{22}}} + \frac{c_{11}^{2}}{\sqrt{a_{11}}}+ \frac{c_{11}^{2}}{\sqrt{a_{22}}} - \frac{c_{12}^{2}}{\sqrt{a_{22}}} + \frac{c_{12}^{2}}{\sqrt{a_{11}}}\bigg) >0.
\end{align}
We obtain the strict inequality of (\ref{squarecomp}) from $c_{11}c_{22}>c_{12}^{2}$ from the strict convexity of $p_{+}$. \end{proof}

Then we have the following existence of foliation for nonzero mean curvature case as in Proposition 5.24 in \cite{CW1}.
\begin{prop} \label{nonzeromcfoliation}
For $g_{0} = (a_{ij})_{1 \le i,j \le 3}$ with $a_{12}=a_{21}=0$, suppose $a_{11} \neq a_{22}$. Then there exists a foliation $\Sigma_{s,t}$ near $p_{+}$ such that $H_{s,t}>0$ and $\rho_{s,t}(x) > \overline{\rho}_{s,t}(x)$ for each $x \in \Sigma_{s,t} \cap \partial M$.
\begin{proof}
By Proposition \ref{positivemeancurv}, we have $H_{p_{+}}>0$. By Proposition \ref{cadifference}, We can find small $t_{0}, s_{0}>0$ such that for $0\le t \le t_{0}$ and $0 \le s \le s_{0}$ such that $H_{s,t}\ge0$ and $\rho_{s,t}(x) > \overline{\rho}_{s,t}(x)$ for each $x \in \Sigma_{s,t} \cap \partial M$. 
\end{proof}
\end{prop}

Let us consider the case $a_{11} = a_{22}$. Since our boundary comparison is scaling invariant, we denote $g'= a_{11}^{-1}g$. Notice that the arguments in the construction of the foliation only relies on $g_{0}|_{\{x_{3}=0\}}=g_{Eucl}|_{\{x_{3}=0\}}$ and mean curvature estimate in the sense of Proposition \ref{mcestimate} and its evolutions, and their work rely neither on $\mathbb{S}^{1}$-symmetry nor their metric comparison condition e.g. $g \ge g_{Eucl}$. Hence, we can apply Proposition 5.28 in \cite{CW1} to $g'$. We denote our new foliation to be $\Sigma'_{t}$ only to avoid confusion with previous foliations.

\begin{prop} \label{zeromcfoliation}
    For $g_{0} = (a_{ij})_{1 \le i,j \le 3}$ with $a_{12}=a_{21}=0$, suppose $a_{11}=a_{22}$. Then there exists a nonnegative constant mean curvature foliation $\Sigma'_{t}$ with prescribed angle $\overline{\rho}$ along $\partial \Sigma'_{t}$.
\begin{proof}
    Let us endow metric $g'$ on $M$ and apply Proposition 5.28 and obtain a CMC foliation near $p_{+}$. This foliation satisfies nonnegative mean curvature condition and prescribed angle condition since we only applied the scaling transformation.
\end{proof}
\end{prop}
\begin{prop}\label{barriergivesminimizer} Suppose there exists a foliation constructed as in Proposition \ref{nonzeromcfoliation} and \ref{zeromcfoliation}. Then there is a nontrivial minimizer of $A^{\overline{\rho}}$ among $\mathcal{I}$.
\begin{proof} 

 Let 
\[
t_0:=\sup_{t\in (0,\epsilon)}\{t: \forall\, \tau<t,\quad \sup_{x \in \Sigma_{\tau}' } H(x)=0\}.
\]

\textbf{Case 1:} Let us assume $t_0=0$, then by the fact that $\Sigma_t'$ has prescribed angle whose lower bound is  $\bar{\rho}$ at each point along $\partial\Sigma_t'$ and the first Variation formula (\ref{firstv}), we have
\[
\frac{d}{dt}F(\Sigma'_t)\big\vert_{t=0}<0,
\]
which implies that the trivial domain at axis point $p_{+}$ is not a minimizer of the capillary functional and there exists a non-trivial minimizer.

\textbf{Case 2:} Now we consider the case of $t_0>0$, then for $0<t<\frac{t_0}{2}$, the foliation $\Sigma_t'$ has a zero mean curvature and prescribed angle $\bar{\rho}$ along $\partial\Sigma_t'$. Moreover, for each $t$, the capillary functional of $\Sigma_t'$ is zero and each slice $\Sigma_t'$ is a nontrivial minimizer of $A^{\overline{\rho}}$. If zero is not the minimal value of the capillary functional among $\mathcal I$, then there has to be a non-trivial minimizer with a negative value of $A^{\overline{\rho}}$, hence we can ensure the existence of nontrivial minimizer in both cases.

Summarizing the above two cases, we conclude that a non-trivial minimizer of the capillary functional always exists. 
\end{proof}
\end{prop}
\begin{proof}[Proof of Main Theorem: smooth case]

By Proposition \ref{nonzeromcfoliation} and \ref{zeromcfoliation}, there exists a non-negative constant mean curvature foliation $\Sigma'_{t}$ with prescribed angle $\bar{\rho}$ along $\partial \Sigma_t'$ in the small neighborhood of $p_{+}$. Now we apply Proposition \ref{barriergivesminimizer}, and we denote a nontrivial minimizer of $A^{\overline{\rho}}$ as $\Sigma$. By Corollary \ref{cor:infinitesmally}, we can construct a zero mean curvature foliation $\Sigma_t$ near $\Sigma$ with prescribed angle $\bar{\rho}$ along $\partial\Sigma_t$.

It is now straightforward to check that the unit normal vector field on $\Sigma$ is parallel (see \cite{bray2010rigidity} or \cite {micallef2015splitting}) and thus the metric is Euclidean. The remaining part is to prove that each slice of $\Sigma$ is congruent to the corresponding slice in the Euclidean model up to scaling.

By Theorem \ref{ineqproof}, we know 
\begin{align*}
    \kappa_{\partial\Sigma_t}&= \frac{H_{\partial M}}{\sin\bar{\rho}}+\frac{1}{\sin^2\bar{\rho}}\partial_{\nu_t}\cos\bar{\rho}\\
    &=\frac{\bar{H}_{\partial M}}{\sin\bar{\rho}}|\nu_t|_{g_{Eucl}}+\frac{|\nu_t|_{g_{Eucl}}}{\sin^2\bar{\rho}}\partial_{\bar{\nu}_t}\cos\bar{\rho}\\
    &=|\nu_t|_{g_{Eucl}}\bar{\kappa}_{\partial \bar{\Sigma}_t}.
\end{align*}

Note that we already proved $g$ is Euclidean up to scaling, then $|\nu_t|_{g_{Eucl}}$ must be a constant for any $t$. Here, $g_{Eucl}$ denotes the Euclidean metric in a model space. We hence proved that each slice is congruent to the corresponding slice of the model case up to a scaling factor and thus the isometry follows. By the connectedness of the domain, we extend the local splitting result to the global splitting result and complete the proof.

\end{proof}

\begin{rem}
    The CMC construction only uses the local comparison near the point $p_+$ which gives the local splitting result. However, to obtain the global splitting result, we need the comparison assumption to hold on every point on the boundary and the connectedness of the boundary.
\end{rem}
\appendix \section{Computations on Euclidean coordinate} Under the parametrization of the boundary of a weakly convex Euclidean domain $\psi(u,v)  = (x(u,v),y(u,v),v)$ where $(u,v) \in S^{1} \times [v_{-},v_{+}] $ and $x_{u}^{2} +y_{u}^{2} = \text{const}$ on level sets $\partial M \cap \{ z= v \}$ for each $v \in [v_{-},v_{+}]$. Denote two tangent vector to be $X_{u}$ and $X_{v}$ by $u$ and $v$ parameter, respectively. Denote $\overline{\tau}$ and $\overline{\nu}$ to be the unit tangential vector field on each level set in counterclockwise direction and the unit normal vector field in $\partial M$. Note that $\overline{\tau}$ and $\overline{\nu}$ consist of orthonormal basis of $T_{p} \partial M$ at each $p \in \partial M$. 
\begin{prop} \label{mcformula} We have the mean curvature formula as follows.
\begin{align}
    H_0+\frac{1}{\sin\bar{\rho}}\nabla_{\bar{\nu}}\cos\bar{\rho}=\sin\bar{\rho}\bar{k}(\gamma_v(u)),
\end{align}
at each point $q=\phi(u,v)$. $\gamma_v(u)=\phi(u,v)$ for fixed $v$, and it is a smooth curve parameterized by $u$, $\bar{k}$ is the geodesic curvature of $\gamma_v(u)$ at point $q$ w.r.t. Euclidean metric.
\end{prop}

\begin{proof}
    For each fixed value $v$,  we define $\bar{\eta}$ as the unit normal vector field along the curve $\gamma_v(u)$, and it lies in the tangent space of the level set $\partial M\cap\{z=v\}$. For each point in $\partial M$, the vector field $\bar{\eta}$ is well-defined, and thus we regard it as a vector field along $\partial M$.

    We rewrite the unit normal vector field $\bar{X}$ as
    \[
    \bar{X}(u,v)=\cos\bar{\rho}e_z+\sin\bar{\rho}\bar{\eta}(u,v).
    \]
    Before computing $H_0$, we prove the following Lemma
    \begin{lem}\label{lem:divergenceofeta}
        \[
        \nabla^{\partial M}\cdot(\bar{\eta}(u,v))=\bar{k}(\gamma_v(u)).
        \]
    \end{lem}
    \begin{proof}[Proof of Lemma \ref{lem:divergenceofeta}]
        By definition, 
        \begin{align*}
            \nabla^{\partial M}\cdot(\bar{\eta})&=\langle \nabla_{\bar{\tau}}\bar{\eta},\bar{\tau}\rangle+\langle\nabla_{\bar{\nu}}\bar{\eta},\bar{\nu}\rangle=\bar{k}(\gamma_v(u))+\langle\nabla_{\bar{\nu}}\bar{\eta},\bar{\nu}\rangle.
        \end{align*}

    Note that
    \[
    \bar{\eta}=\sin\bar{\rho}\bar{X}+\cos\bar{\rho}\bar{\nu},
    \]
    we compute
    \begin{align*}
        \langle\nabla_{\bar{\nu}}\bar{\eta},\bar{\nu}\rangle &=\langle\nabla_{\bar{\nu}}(\sin\bar{\rho}\bar{X}+\cos\bar{\rho}\bar{\nu}),\bar{\nu}\rangle\\
        &=\sin\bar{\rho}\RN{2}(\bar{\nu},\bar{\nu})+\nabla_{\bar{\nu}}\langle \bar{X}, e_z\rangle\\
        &=\sin\bar{\rho}\RN{2}(\bar{\nu},\bar{\nu})-\sin\bar{\rho}\RN{2}(\bar{\nu},\bar{\nu})\\
        &=0.
    \end{align*}
   \end{proof}    
   Now, we prove the statement 
   \begin{align*}
       H_0 =&\nabla^{\partial M}\cdot(\bar{X})\\
           =&\nabla^{\partial M}\cdot(\cos\bar{\rho}e_z+\sin\bar{\rho}\bar{\eta})\\
       =&\langle\nabla\cos\bar{\rho},e_z\rangle-\cot\bar{\rho}\langle\nabla\cos\bar{\rho},\bar{\eta}\rangle+\sin\bar{\rho}\nabla^{\partial M}\cdot(\bar{\eta})\\
       =&-\frac{1}{\sin\bar{\rho}}\langle\nabla\cos\bar{\rho},\bar{\nu}\rangle+\sin\bar{\rho}\bar{k}(\gamma_v(u)).
   \end{align*}
   We used Lemma \ref{lem:divergenceofeta} in the last equality.
\end{proof}

We also prove the following comparison Lemma.
\begin{lem}\label{lem:weak convexity}
    For each point $p\in\partial M$, we have
    \[
    \left(H_0(p)-\nabla_{\bar{\nu}}\bar{\rho}\right)\nabla_{\bar{\nu}}\bar{\rho}-(\nabla_{\bar{\tau}}\bar{\rho})^2\geq0.
    \]
\end{lem}
\begin{proof}
    Note that
    \begin{align*}
        \nabla_{\bar{\nu}}\cos\bar{\rho}&=\langle\nabla_{\bar{\nu}}\bar{X},e_z\rangle=-\sin\bar{\rho}\RN{2}(\bar{\nu},\bar{\nu}).
    \end{align*}
    We have 
    \[
    H_0-\nabla_{\bar{\nu}}\bar{\rho}=\RN{2}(\bar{\tau},\bar{\tau}).
    \]
    Here we used the definition of mean curvature.
    
    Similarly, we compute
    \begin{align*}
        \nabla_{\bar{\tau}}\cos\bar{\rho}&=\langle \nabla_{\bar{\tau}}\bar{X},e_z\rangle=-\sin\bar{\rho}\RN{2}(\bar{\tau},\bar{\nu}).
    \end{align*}
    Combining the above computations, we obtain
    \begin{align*}
        \left(H_0(p)-\nabla_{\bar{\nu}}\bar{\rho}\right)\nabla_{\bar{\nu}}\bar{\rho}-(\nabla_{\bar{\tau}}\bar{\rho})^2&\geq \RN{2}{(\bar{\nu},\bar{\nu})}\RN{2}{(\bar{\tau},\bar{\tau})}-\RN{2}(\bar{\tau},\bar{\nu})^2\\
        &=\det(\RN{2})\\
        &\geq 0
    \end{align*}
    The last inequality follows from the weak convexity assumption.
\end{proof}

\bibliographystyle{plain}
\bibliography{reference}

\end{document}